\RequirePackage{fix-cm}
\RequirePackage{fixltx2e}
\documentclass[12pt,english]{article}
\usepackage{ae,aecompl}
\usepackage[T1]{fontenc}
\usepackage[latin9]{inputenc}
\usepackage{geometry}
\usepackage{babel}
\usepackage{enumitem}
\usepackage{endnotes}
\usepackage{amsmath}
\usepackage{amsthm}
\usepackage{amssymb}
\usepackage[unicode=true,pdfusetitle,
 bookmarks=true,bookmarksnumbered=false,bookmarksopen=false,
 breaklinks=false,pdfborder={0 0 1},backref=false,colorlinks=false]
 {hyperref}

\makeatletter

\numberwithin{equation}{section}
\numberwithin{figure}{section}
 \usepackage{changebar}
 \providecommand{\lyxadded}[3]{}
 
 \renewcommand{\lyxadded}[3]{
   {\protect\cbstart\color{lyxadded}{}#3\protect\cbend}
 }
 
 \IfFileExists{candleshape.def}{%
  \input{candleshape.def}}{}
 \IfFileExists{dropshape.def}{%
  \input{dropshape.def}}{}
 \IfFileExists{TeXshape.def}{%
  \input{TeXshape.def}}{}
 \IfFileExists{triangleshapes.def}{%
  \input{triangleshapes.def}}{}

 \let\footnote=\endnote
\@ifundefined{lettrine}{\usepackage{lettrine}}{}
\theoremstyle{plain}
\newtheorem{theorem}{\protect\theoremname}
  \theoremstyle{remark}
  \newtheorem{remark}[theorem]{\protect\remarkname}
    \theoremstyle{lemma}
  \newtheorem{lemma}[theorem]{\protect\lemmaname}
  \theoremstyle{definition}
  \newtheorem{definition}[theorem]{\protect\definitionname}

\usepackage{color}
\usepackage[]{graphicx}
\usepackage[caption = false]{subfig}
\makeatother

\providecommand{\definitionname}{Definition}
\providecommand{\remarkname}{Remark}
\providecommand{\theoremname}{Theorem}
\providecommand{\lemmaname}{Lemma}

\theoremstyle{plain}

\begin{document}

\title{Correctors justification for a Smoluchowski--Soret--Dufour model posed in perforated domains}

\author{Vo Anh Khoa\thanks{Author for correspondence. Mathematics and Computer Science Division,
Gran Sasso Science Institute, L'Aquila, Italy. (khoa.vo@gssi.infn.it,
vakhoa.hcmus@gmail.com)}, and Adrian Muntean\thanks{Department of Mathematics and Computer Science, Karlstad University,
Sweden. (adrian.muntean@kau.se)}}
\maketitle
\begin{abstract}
We study a coupled thermo-diffusion system that accounts for the dynamics of hot colloids in periodically heterogeneous media. Our model describes the joint evolution of temperature and colloidal concentrations in a saturated porous structure, where the Smoluchowski interactions are responsible for aggregation and fragmentation processes in the presence of 
Soret-Dufour type effects. Additionally, we allow for deposition and depletion on internal micro-surfaces. In this work, we derive corrector estimates quantifying the rate of convergence of the periodic homogenization limit process performed in \cite{KAM14} via two-scale convergence arguments. The major technical difficulties in the proof are linked to the estimates between nonlinear processes of aggregation and deposition and to the convergence arguments of the \emph{a priori} information of the oscillating weak solutions and cell functions in high dimensions. Essentially, we circumvent the arisen  difficulties by a suitable use of the energy method and of fine integral estimates controlling interactions at the level of micro-surfaces.
\end{abstract}

\section{Introduction}

Diffusion and heat conduction, taken separately, are well understood processes at a large variety of space scales. However, as soon as diffusion interplays with the conduction of heat, it appears that the structure of the model equations is not so clear as one would expect, especially if one wants to describe settings away from the somewhat better understood thermodynamic equilibrium, where statistical mechanics is the main investigation tool.

Driven by possible applications in the context of efficient drug-delivery and in the design of intelligent packaging materials,  we wish to understand mathematically the upscaling of  the following basic thermo-diffusion scenario: We look at a population of colloidal particles (monomers) driven by a flux linearly combining Fick and Fourier contributions.  We assume that monomers undergo a Smoluchowski-like dynamics producing populations of $i$-mers that finally meet and travel through a transversal porous membrane.  The microscopic boundaries (at the level of the membrane pores) are  active in the sense that they host adsorption and desorption of clusters of colloidal particles.  

The starting PDE model is formulated 
in \cite{KAM14} by Krehel and his co-authors. Their  thermo-diffusion system is posed in perforated media with uniform periodicity inside the domain. As main outcome, they prove both the global weak  solvability of the model as well as the periodic homogenization limit. As byproduct, they also obtain the precise structure of the effective transport parameters. Now, is the moment to: 
Justify the two-scale asymptotics by proving corrector/error estimates for the homogenization limit for periodic arrangements of membrane pores/microstructures.

In our context, the structure of the corrector estimate for the involved concentrations and temperature fields  we wish to prove is
\[
\left\Vert \theta^{\varepsilon}-\theta_{0}^{\varepsilon}\right\Vert _{L^{2}\left((0,T)\times\Omega^{\varepsilon}\right)}^{2}+\left\Vert u^{\varepsilon}-u_{0}^{\varepsilon}\right\Vert _{L^{2}\left((0,T)\times\Omega^{\varepsilon}\right)}^{2}
\]
\begin{equation}
+\left\Vert \nabla\left(\theta^{\varepsilon}-\theta_{1}^{\varepsilon}\right)\right\Vert _{L^{2}\left(\left(0,T\right)\times\Omega^{\varepsilon}\right)}^{2}+\left\Vert \nabla\left(u^{\varepsilon}-u_{1}^{\varepsilon}\right)\right\Vert _{L^{2}\left(\left(0,T\right)\times\Omega^{\varepsilon}\right)}^{2}+\varepsilon\left\Vert v^{\varepsilon}-v_{0}^{\varepsilon}\right\Vert _{L^{2}\left((0,T)\times\Gamma^{\varepsilon}\right)}^{2}\le C\varepsilon,
\label{eq:pseudores}
\end{equation}
where $C>0$ is a generic constant independent of the choice of the scale parameter $\varepsilon>0$.

To obtain this corrector estimate, our strategy is to use an energy-like method  and macroscopic reconstructions (cf. e.g.\cite{Eck2}, but also \cite{EM17}). This technique basically relies on the choice of  test functions able to capture  in suitable norms the difference between the micro-and macro-concentrations as well as  micro- and macro-temperatures and their transport fluxes. Careful attention needs to be payed to the regularity  of the limit solutions as well as of the  cell functions  involved in the asymptotic procedure; see e.g. \cite{KM16,Ptash}. A similar approach has been followed by Eck et al. (cf. e.g. \cite{Eck1,Eck2}) concerning  the upscaling of  the phase field model in high contrast regimes. Besides handling new nonlinear terms, the novel aspect in our context is the handling of the errors produced in the upscaling due to micro-surfaces. A similar analysis can be carried over the settings in \cite{AP06,RMK12,SPK14,FM14}, e.g.

Besides the energy-like approach used here for a periodic homogenization case, powerful contributions can be obtained using variants of the bulk and boundary unfollding operators: see, for instance, \cite{GRi04,Vernescu07,Ptash,MR16}. Using somewhat more regularity,  high-order corrector estimates can be obtained  for semi-linear elliptic systems via an iteration method that uses explicitly the expected structure of the two-scale asymptotic expansion; compare \cite{KM16,Khoa17}. Settings involving locally-periodic microstructures can be treated as in \cite{Tycho}, e.g., while the random case is in most of the cases out of reach; see \cite{Kozlov79,PY90} for some details in this direction.

Having available corrector estimates like (\ref{eq:pseudores}) allows in principle the construction of convergence proofs as well as \emph{a priori} error estimate for MsFEM applied to problems in perforated media like in \cite{Legoll2014}, for instance.

This paper is structured as follows: Section \ref{sec:2} is devoted to the presentation of  the Smoluchowski-Sorect-Dufour model  posed in a perforated domain. In this section, we also list a couple of  preliminary results  about that the two-scale convergence and compactness arguments and  about the weak solvability of both the microscopic and limit models (recalling from \cite{KAM14}).  Our main result is Theorem \ref{mainthm:1}, as presented in Section \ref{sec:3}. We then introduce the derivation of the difference system resulting from the microscopic problem and the "macroscopic reconstructed" system. On top of that, we prepare in this part  a few helpful integral estimates. The proof of  Theorem \ref{mainthm:1} is provided in Section \ref{sec:4}. We conclude the  paper with the remarks from Section \ref{sec:conclu}.


\section{Setting of the problem}\label{sec:2}
\subsection{The coupled thermo-diffusion model}
\subsubsection{A geometrical interpretation of porous medium}\label{subsec:geometry}

Let $\Omega$ be a bounded open domain in $\mathbb{R}^{d}$ ($d\in\left\{ 2,3\right\} $) 
with $\partial\Omega\in C^{0,1}$. 
Without
loss of generality, we reduce ourselves to consider $\Omega$ as the
parallelepiped $\left(0,a_{1}\right)\times...\times\left(0,a_{d}\right)$
with $a_{i}>0,i\in\left\{ 1,...,d\right\} $. Let $Y$ be the
representative unit cell defined by
\[
Y:=\left\{ \sum_{i=1}^{d}\lambda_{i}\vec{e}_{i}:0<\lambda_{i}<1\right\} ,
\]
where $\vec{e}_{i}$ is the $i$th unit vector in $\mathbb{R}^{d}$.

Let $Y_{0}$ be an open subset of $Y$ with a Lipschitz boundary
$\Gamma = \partial Y_{0}$ which is divided into two disjoint closed parts $\Gamma_{N}$ and $\Gamma_{R}$ with a nonzero $(d-1)$-dimensional measure, i.e. $\Gamma = \Gamma_{N}\cup \Gamma_{R}$ with $\Gamma_{N}\cap\Gamma_{R} = \emptyset$.

Let $Z\in \mathbb{R}^{d}$ be a hypercube. Then for $X\subset Z$ we denote by $X^k$ the shifted subset
\[
X^{k}:=X+\sum_{i=1}^{d}k_{i}\vec{e}_{i},
\]
where $k=\left(k_{1},...,k_{d}\right)\in\mathbb{Z}^{d}$ is a vector
of indices.

Assume that a scale factor $\varepsilon > 0$ is given. The pore skeleton is then defined as
the union of $\varepsilon Y_{0}^{k}$ the $\varepsilon$-homothetic
sets of $Y_{0}^{k}$, i.e.
\[
\Omega_{0}^{\varepsilon}:=\bigcup_{k\in\mathbb{Z}^{d}}\left\{ \varepsilon Y_{0}^{k}:Y_{0}^{k}\subset\Omega\right\} .
\]

Thus, the total pore space we have in mind is $\Omega^{\varepsilon}=\Omega\backslash\Omega_{0}^{\varepsilon}$.

Set $Y_{1}:=Y\backslash\overline{Y_{0}}$. The unit cell $Y$ is made of two
parts including the gas phase $Y_{1}$ and the solid phase $Y_{0}$.
We denote the total pore surface of the skeleton by $\Gamma^{\varepsilon}:=\partial\Omega_{0}^{\varepsilon}$. The pore surface $\Gamma^{\varepsilon}$ consists of two parts satisfying
$\Gamma^{\varepsilon}=\Gamma_{N}^{\varepsilon}\cup\Gamma_{R}^{\varepsilon}$
where $\Gamma_{N}^{\varepsilon}$ and $\Gamma_{R}^{\varepsilon}$
are disjoint closed sets
possessing a nonzero $\left(d-1\right)$-dimensional measure. The
Neumann boundary $\Gamma_{N}^{\varepsilon}$ indicates the insulation
for the heat flow, whilst at $\Gamma_{R}^{\varepsilon}$ we allow for a
flux of mass through a Robin-type condition. The union of
the cell regions $\varepsilon Y_{1}^{k}$ (without the solid grains
$\varepsilon Y_{0}^{k}$) represents the total available space
for thermo-diffusion. 

In Figure \ref{fig:1} and Figure \ref{fig:2}, we show a  admissible 2d domain with microstructures.
We let throughout the paper $\mbox{n}:=\left(n_{1},...,n_{d}\right)$
be the unit outward normal vector on the boundary $\partial\Omega^{\varepsilon}$. The representation of the periodic geometries is inspired from \cite{HJ91,KM16,RMK12} and references cited therein, but other possibilities exist as well. The practical problem usually delimitates the freedom in choosing the precise structure of $Y_0$; see Figure \ref{fig:2} for a couple of options.

\begin{figure}[!h]
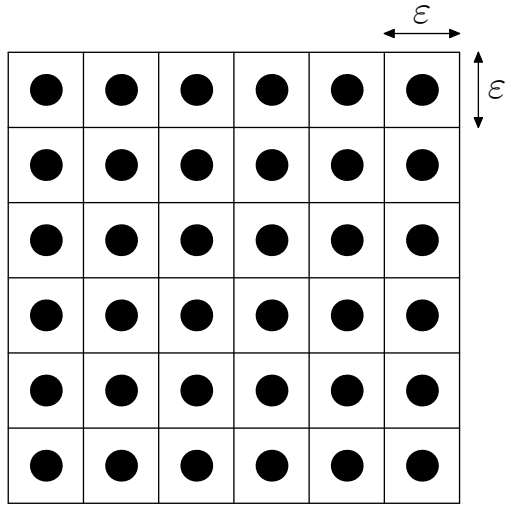

	\centering	
	\parbox{10cm}{\convertMPtoPDF{fig_grid.eps}{1.0}{1.0}}	
	\caption{An admissible 2d perforated domain.}
	\label{fig:1}
\end{figure}
\begin{figure}[!h]
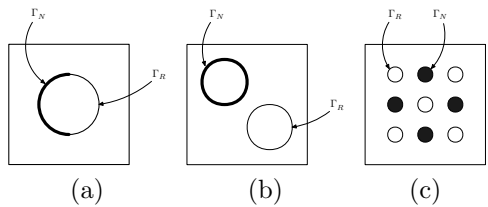

	\centering
	\subfloat[][]{
		\parbox{10cm}{\convertMPtoPDF{what_01.eps}{0.4}{0.4}}	
	}
	\subfloat[][]{
		\parbox{10cm}{\convertMPtoPDF{what_02.eps}{0.4}{0.4}}	
	}
	\subfloat[][]{
		\parbox{10cm}{\convertMPtoPDF{what_03.eps}{0.4}{0.4}}	
	}	
	\caption{Possible choices for $Y_0$. The choice of (a) fits to the geometry described in Figure \ref{fig:1}.}
	\label{fig:2}
\end{figure}


\subsubsection{Model description}

Before describing the microscopic problem (which we refer to as $(P^{\varepsilon})$), we define some useful notation. For $\delta>0$, let $\nabla^{\delta}$ be the so-called mollified gradient
\[
\nabla^{\delta}f\left(x\right):=\nabla\left[\int_{B\left(x,\delta\right)}J_{\delta}\left(x-y\right)f\left(y\right)dy\right],
\]
where $J_{\delta}$ is a mollifier (see e.g. \cite{Evan98}) and $B\left(x,\delta\right)$ is the ball centered in $x\in \Omega$
with radius $\delta$. The radius $\delta$ is assumed to be an $\varepsilon$-independent constant.

We denote by $x\in\Omega^{\varepsilon}$ the macroscopic variable
and by $y=x/\varepsilon$ the microscopic variable representing fast
variations at the microscopic geometry. With this convention, we write
\[
\kappa^{\varepsilon}\left(x\right)=\kappa\left(\frac{x}{\varepsilon}\right)=\kappa\left(y\right).
\]

The same convention applies to all the other oscillating coefficients involved our problem.

We denote by $\mathcal{A}^{\varepsilon}_{\mathbb{T}}$ the second-order elliptic
operator in divergence form with rapidly oscillating coefficients,
i.e.
\begin{equation}
\mathcal{A}_{\mathbb{T}}^{\varepsilon}:=\nabla\cdot\left(-\mathbb{T}\left(\frac{x}{\varepsilon}\right)\nabla\right)=\frac{\partial}{\partial x_{i}}\left[-\tau_{ij}^{\alpha\beta}\left(\frac{x}{\varepsilon}\right)\frac{\partial}{\partial x_{j}}\right].\label{eq:2.8}
\end{equation}

Concerning the structure of $\mathcal{A}^{\varepsilon}_{\mathbb{T}}$, we assume that for all $y\in Y$, $\mathbb{T}\left(y\right)=\left(\tau_{ij}^{\alpha\beta}\left(y\right)\right):\mathbb{R}^{d}\to\mathbb{R}^{m^{2}\times d^{2}}$
for $1\le i,j\le d,1\le\alpha,\beta\le m$ is a second-order tensor that depends
on the position vector $y$ and satisfies a uniform (in $\varepsilon$) ellipticity
condition. Depending on the situation, we have either $\mathbb{T}$ is the tensor $\kappa$  (heat conductivity) or the tensor $d_i$ (diffusion coefficients). Note that $m\ge1$ denotes the number of balance equations in the
system.

In this framework, we consider that maximum $N>2$ colloidal species are involved in the thermo-diffusion process. We denote by $\left(\theta^{\varepsilon},u_{i}^{\varepsilon},v_{i}^{\varepsilon}\right)$
for $i\in\left\{ 1,...,N\right\} $ the triplet of real-valued solutions
of our thermo-diffusion model, i.e. a system of 
coupled ordinary differential equations with semi-linear parabolic equations for the evolution of temperature
and colloid concentrations. Denote by $u^{\varepsilon}:=\left(u_{1}^{\varepsilon},...,u_{N}^{\varepsilon}\right)$  the vector of all active colloidal concentrations $u_{i}^{\varepsilon}$. We assume that these species obey the population balance equation as postulated by Smoluchowski
in \cite{Smo17}, i.e.
\[
R_{i}\left(s\right):=\frac{1}{2}\sum_{k+j=i}\beta_{kj}s_{k}s_{j}-\sum_{j=1}^{N}\beta_{ij}s_{i}s_{j},\quad (\text{with}\;R_{i}:\mathbb{R}^{N}\to\mathbb{R},i\in \left\{1,...,N\right\})
\]
theoretically representing a quadratic-like rate of change of $s_{i}$.
The presence of coagulation coefficients $\beta_{ij}>0$ accounts
for the rate aggregation and fragmentation between populations of particles of size
$i$ and $j$. For further modeling details, we refer the reader to \cite{EGJW98,GKW11,GW11}
and \cite{KMK15}, e.g.

We denote the parabolic cylinders as $Q_{T}^{\varepsilon}:=\left(0,T\right)\times\Omega^{\varepsilon}$ and $Q_T:=(0,T)\times\Omega$. Now, we detail the structure of our microscopic problem ($P^{\varepsilon}$).
For $i\in \left\{1,...,N\right\}$, we consider the following coupled thermo-diffusion system:
\begin{equation}
\partial_{t}\theta^{\varepsilon}+\mathcal{A}_{\kappa}^{\varepsilon}\theta^{\varepsilon}=\tau^{\varepsilon}\sum_{i=1}^{N}\nabla^{\delta}u_{i}^{\varepsilon}\cdot\nabla\theta^{\varepsilon}\quad\mbox{in}\;Q_{T}^{\varepsilon},
\label{eq:eq1}
\end{equation}
\begin{equation}
\partial_{t}u_{i}^{\varepsilon}+\mathcal{A}_{d_{i}}^{\varepsilon}u_{i}^{\varepsilon}=\rho_{i}^{\varepsilon}\nabla^{\delta}\theta^{\varepsilon}\cdot\nabla u_{i}^{\varepsilon}+R_{i}\left(u^{\varepsilon}\right)\quad\mbox{in}\;Q_{T}^{\varepsilon},\label{eq:eq2}
\end{equation}
\begin{equation}
\partial_{t}v_{i}^{\varepsilon}=a_{i}^{\varepsilon}u_{i}^{\varepsilon}-b_{i}^{\varepsilon}v_{i}^{\varepsilon}\quad\mbox{on}\;\left(0,T\right)\times\Gamma^{\varepsilon},\label{eq:odedeposit}
\end{equation}
subject to the boundary conditions
\begin{equation}
-\kappa^{\varepsilon}\nabla\theta^{\varepsilon}\cdot\mbox{n}=0\quad\mbox{on}\;\left(0,T\right)\times\Gamma_{N}^{\varepsilon},\label{eq:boundaryNeu}
\end{equation}
\begin{equation}
-\kappa^{\varepsilon}\nabla\theta^{\varepsilon}\cdot\mbox{n}=\varepsilon g_0^{\varepsilon} \theta^{\varepsilon}\quad\mbox{on}\;\left(0,T\right)\times\Gamma_{R}^{\varepsilon},\label{eq:robinbound}
\end{equation}
\begin{equation}
-\kappa^{\varepsilon}\nabla\theta^{\varepsilon}\cdot\mbox{n}=0\quad\mbox{on}\;\left(0,T\right)\times\partial\Omega,
\end{equation}
\begin{equation}
-d_{i}^{\varepsilon}\nabla u_{i}^{\varepsilon}\cdot\mbox{n}=\varepsilon\left(a_{i}^{\varepsilon}u_{i}^{\varepsilon}-b_{i}^{\varepsilon}v_{i}^{\varepsilon}\right)\quad\mbox{on}\;\left(0,T\right)\times\Gamma^{\varepsilon},\label{eq:deposit}
\end{equation}
\begin{equation}
-d_{i}^{\varepsilon}\nabla u_{i}^{\varepsilon}\cdot\mbox{n}=0\quad\mbox{on}\;\left(0,T\right)\times\partial\Omega,
\end{equation}
and the initial data
\begin{equation}
\theta^{\varepsilon}\left(0,x\right)=\theta^{\varepsilon,0}\left(x\right)\quad\mbox{for}\;x\in\Omega^{\varepsilon}.\label{eq:2.11}
\end{equation}
\begin{equation}
u_{i}^{\varepsilon}\left(0,x\right)=u_{i}^{\varepsilon,0}\left(x\right)\quad\mbox{for}\;x\in\Omega^{\varepsilon},\label{eq:2.12}
\end{equation}
\begin{equation}
v_{i}^{\varepsilon}\left(0,x\right)=v_{i}^{\varepsilon,0}\left(x\right)\quad\mbox{for}\;x\in\Gamma^{\varepsilon}.\label{eq:2.13}
\end{equation}

\eqref{eq:eq1}-\eqref{eq:2.13} form our microscopic problem $(P^{\varepsilon})$.

\begin{table}[h]
	\centering
	\caption{Physical parameters of the microscopic problem ($P^{\varepsilon}$).}
	\begin{tabular}{@{}lll@{}} \hline
		$\kappa^{\varepsilon}$   & heat conductivity (tensor)\\
		$\tau^{\varepsilon}$   & Soret coefficient (tensor)\\
		$g_0^{\varepsilon}$  &  heat absorption (scalar)\\
		$d_{i}^{\varepsilon}$    & diffusion coefficients (tensor)\\
		$\rho_{i}^{\varepsilon}$   & Dufour coefficients (tensor) \\
		$a_i^{\varepsilon},b_i^{\varepsilon}$  & deposition rate coefficients (scalars)\\
		\hline
	\end{tabular}
	\label{tab:1}
\end{table}

\begin{remark}
	Our thermo-diffusion system is made of $N+1$ equations where the short-hand explanation for physical parameters in this model can be found in Table \ref{tab:1}. Physically, equation \eqref{eq:eq1} describes the changes of the temperature  $\theta^{\varepsilon}$
	in $\Omega^{\varepsilon}$ according to a heat conduction equation with a production term depending on $\nabla^{\delta}u_{i}^{\varepsilon}$, whilst the colloidal concentration $u_{i}^{\varepsilon}$
	is assumed to satisfy $N$ reaction-diffusion like equations given
	by \eqref{eq:eq2} with a chemical reaction term depending on $\nabla^{\delta}\theta^{\varepsilon}$. This type of special right-hand sides is mimicking the so-called Soret and Dufour effects. In \eqref{eq:deposit}, $v_{i}^{\varepsilon}$ denotes the mass of
	the deposited species on the boundary of the pore skeleton $\Gamma^{\varepsilon}$.
	These quantities are also supposed to satisfy the following ordinary
	differential equations \eqref{eq:odedeposit}.
\end{remark}


We make use of the following assumptions:

$\left(\mbox{A}_{1}\right)$ The coefficients $\kappa^{\varepsilon},\tau^{\varepsilon},d_{i}^{\varepsilon},\rho_{i}^{\varepsilon}\in [H_{+}^1(\Omega^{\varepsilon})]^{d^2}\cap [L^{\infty}_{+}(\Omega^{\varepsilon})]^{d^2}$, $g_0^{\varepsilon}\in L_{+}^{\infty}(\Gamma_{R}^{\varepsilon})$ and $a_i^{\varepsilon},b_i^{\varepsilon}\in L_{+}^{\infty}\left(\Gamma^{\varepsilon}\right)$ 
are $Y$-periodic. Also, there exist positive constants $\kappa_{\min}$, $\kappa_{\max}$,
$\tau_{\min}$, $\tau_{\max}$, $d_{\min}$, $d_{\max}$, $\rho_{\min}$,
$\rho_{\max},a_{\min},a_{\max},b_{\min},b_{\max}$ such that $\kappa_{\min}\le\kappa_{jk}\le\kappa_{\max},$
$\tau_{\min}\le\tau_{jk}\le\tau_{\max}$, $d_{\min}\le d_{i}^{jk}\le d_{\max}$,
$\rho_{\min}\le\rho_{i}^{jk}\le\rho_{\max}$, $a_{\min}\le a_i^{\varepsilon} \le a_{\max}$, $b_{\min}\le b_i^{\varepsilon} \le b_{\max}$ for $i\in\left\{ 1,...,N\right\} $ and $j,k\in \left\{ 1,...,d\right\}$. Furthermore, there also exist  positive constants $\alpha_{i}$ for $i\in\left\{ 0,...,N\right\} $ such that
\[
\kappa_{jk}\left(y\right)\xi_{j}\xi_{k}\ge\alpha_{0}\left|\xi\right|^{2}\;\text{and}\;d_{i}^{jk}\left(y\right)\xi_{j}\xi_{k}\ge\alpha_{i}\left|\xi\right|^{2}\quad\text{for any}\;\xi\in\mathbb{R}^{d},i\in\left\{ 1,...,N\right\},j\;\text{and}\;k\in\left\{ 1,...,d\right\}
\]
to guarantee the ellipticity of the operators $\mathcal{A}_{\kappa}^{\varepsilon}$ and $\mathcal{A}_{d_i}^{\varepsilon}$.

$\left(\mbox{A}_{2}\right)$ The initial conditions satisfy $\theta^{\varepsilon,0}\in L_{+}^{\infty}\left(\Omega^{\varepsilon}\right)\cap H^{1}\left(\Omega^{\varepsilon}\right)$,
$u_{i}^{\varepsilon,0}\in L_{+}^{\infty}\left(\Omega^{\varepsilon}\right)\cap H^{1}\left(\Omega^{\varepsilon}\right)$,
$v_{i}^{\varepsilon,0}\in L_{+}^{\infty}\left(\Gamma^{\varepsilon}\right)$
for $i\in\left\{ 1,...,N\right\} $, such that we can find $C_{0}>0$
satisfying
\[
\left\Vert \theta^{\varepsilon,0}\right\Vert _{H^{1}\left(\Omega^{\varepsilon}\right)}+\sum_{i=1}^{N}\left(\left\Vert u_{i}^{\varepsilon,0}\right\Vert _{H^{1}\left(\Omega^{\varepsilon}\right)}+\left\Vert v_{i}^{\varepsilon,0}\right\Vert _{L^{\infty}\left(\Gamma^{\varepsilon}\right)}\right)\le C_{0},
\]
where $C_0$ is independent of the choice of $\varepsilon$.

\begin{remark}
	By the definitions of $\kappa,\tau,d_{i},\rho_{i}$ and $\left(\mbox{A}_{1}\right)$, there exist positive constants that bound from below and above these coefficients on $Y$ for each choice of $\varepsilon$.
\end{remark}

Unless otherwise specified, all the constants $C$ are independent
of the homogenization parameter $\varepsilon$, but the precise values may differ from line to line or even within a single chain of estimates. Throughout this paper, we use the superscript $\varepsilon$
to emphasize the dependence on the heterogeneity of the material characterized
by the homogenization parameter $\varepsilon$. In the sequel, we use $dS_{\varepsilon}$ as a shorthand for $\mbox{n}dS_{\varepsilon}$ where $S_{\varepsilon}$ can be viewed as a common notation for a boundary of any surface. Moreover, the notation $\left| \cdot \right|$ for a domain indicates in this work the volume of that domain. 

\subsection{Preliminary results}

In this subsection, we present the definition of two-scale convergence as well as its
compactness arguments (cf. \cite{All92,Ngu89}) together with the fact already known concerning the weak solvability and periodic homogenization of ($P^{\varepsilon}$).

\begin{definition}
	\textbf{Two-scale convergence}
	
	Let $\left(u^{\varepsilon}\right)$ be a sequence of functions in
	$L^{2}\left(0,T;L^{2}\left(\Omega\right)\right)$ with $\Omega$ being
	an open set in $\mathbb{R}^{d}$, then it two-scale converges to a
	unique function $u^{0}\in L^{2}\left(\left(0,T\right)\times\Omega\times Y\right)$,
	denoted by $u^{\varepsilon}\stackrel{2}{\rightharpoonup}u^{0}$, if
	for any $\varphi\in C_{0}^{\infty}\left(\left(0,T\right)\times\Omega;C_{\#}^{\infty}\left(Y\right)\right)$
	we have
	\[
	\lim_{\varepsilon\to0}\int_{0}^{T}\int_{\Omega}u^{\varepsilon}\left(t,x\right)\varphi\left(t,x,\frac{x}{\varepsilon}\right)dxdt=\frac{1}{\left|Y\right|}\int_{0}^{T}\int_{\Omega}\int_{Y}u^{0}\left(t,x,y\right)\varphi\left(t,x,y\right)dydxdt.
	\]
	
\end{definition}

\begin{theorem}
	\emph{\textbf{Two-scale compactness}}
	\begin{itemize}
		\item Let $\left(u^{\varepsilon}\right)$ be a bounded sequence in $L^{2}\left((0,T)\times\Omega\right)$. Then there exists a function $u^{0}\in L^{2}\left(\left(0,T\right)\times\Omega\times Y\right)$ such that, up to a subsequence, $u^{\varepsilon}$ 
		two-scale converges to $u^0$.
		\item Let $\left(u^{\varepsilon}\right)$ be a bounded sequence in $L^{2}\left(0,T;H^{1}\left(\Omega\right)\right)$,
		then up to a subsequence, we have the two-scale convergence in gradient
		$\nabla u^{\varepsilon}\stackrel{2}{\rightharpoonup}\nabla_{x}u^{0}+\nabla_{y}u^{1}$
		for $u^{0}\in L^{2}\left(0,T;H^1(\Omega)\right)$
		and $u^{1}\in L^{2}\left(\left(0,T\right)\times\Omega;H_{\#}^{1}\left(Y\right)/\mathbb{R}\right)$. 
	\end{itemize}
\end{theorem}

\begin{remark}
	The concepts of two-scale convergence and compactness for $\varepsilon$-periodic
	hypersurfaces were originally introduced in \cite{Radu96,AADH95} and have been used in \cite{FM14,KAM14}. For brevity, let $(u^{\varepsilon})$ be a sequence of functions in $L^2(0,T;L^2(\Gamma^{\varepsilon}))$. We say $u^{\varepsilon}$ two-scale converges to a limit $u^0$ in $L^2((0,T) \times\Omega\times \Gamma)$ with $\Gamma = \partial\Omega$ if for any $\varphi\in C_{0}^{\infty}((0,T)\times\Omega; C_{\#}^{\infty}(\Gamma))$ we have
	\[
	\lim_{\varepsilon\to0}\int_{0}^{T}\int_{\Gamma^{\varepsilon}}\varepsilon u^{\varepsilon}\left(t,x\right)\varphi\left(t,x,\frac{x}{\varepsilon}\right)dxdt=\frac{1}{\left|Y\right|}\int_{0}^{T}\int_{\Omega}\int_{\Gamma}u^{0}\left(t,x,y\right)\varphi\left(t,x,y\right)d\sigma(y)dxdt.
	\]
	Thereby, we obtain the two-scale compactness on surfaces that for each bounded sequence $\left(u^{\varepsilon}\right)$ in $L^{2}\left(0,T;L^{2}\left(\Gamma^{\varepsilon}\right)\right)$, one can extract a subsequence which two-scale converges to $u^0\in L^2((0,T) \times\Omega\times \Gamma)$. Furthermore, if $\left(u^{\varepsilon}\right)$ is bounded in $L^{\infty}\left(0,T;L^{\infty}\left(\Gamma^{\varepsilon}\right)\right)$, it then two-scale converges to a limit function $u^0\in L^{\infty}((0,T) \times\Omega\times \Gamma)$.
	
	It is important to note that, for our choice of $Y_0$, the interior
	extension from $H^{1}\left(\Omega^{\varepsilon}\right)$ into $H^{1}\left(\Omega\right)$
	exists with extension constants independent of $\varepsilon$
	(see \cite[Lemma 5]{HJ91} and \cite[Theorem 2.10]{CP99}).
\end{remark}

\begin{definition}
	\label{def:mathbb=00007BPeps=00007D}\textbf{The weak formulation of $(P^{\varepsilon})$}
	
	For $i\in\left\{ 1,...,N\right\} $, the triplet $\left(\theta^{\varepsilon},u_{i}^{\varepsilon},v_{i}^{\varepsilon}\right)$
	satisfying
	\[
	\theta^{\varepsilon},u_{i}^{\varepsilon}\in
	H^{1}\left(0,T;L^{2}\left(\Omega^{\varepsilon}\right)\right)\cap
	L^{\infty}\left(0,T;H^{1}\left(\Omega^{\varepsilon}\right)\right)\cap L^{\infty}\left(\left(0,T\right)\times\Omega^{\varepsilon}\right),
	\]
	\[
	v_{i}^{\varepsilon}\in
	H^{1}\left(0,T;L^{2}\left(\Gamma^{\varepsilon}\right)\right)\cap L^{\infty}\left(\left(0,T\right)\times\Gamma^{\varepsilon}\right).
	\]
	is a weak solution to $\left(P^{\varepsilon}\right)$ provided that
	\begin{equation}
	\left\{ \begin{array}{c}
	{\displaystyle \int_{\Omega^{\varepsilon}}\partial_{t}\theta^{\varepsilon}\varphi dx+\int_{\Omega^{\varepsilon}}\kappa^{\varepsilon}\nabla\theta^{\varepsilon}\cdot\nabla\varphi dx+\varepsilon\int_{\Gamma_{R}^{\varepsilon}}g_{0}\theta^{\varepsilon}\varphi dS_{\varepsilon}=\int_{\Omega^{\varepsilon}}\tau^{\varepsilon}\sum_{i=1}^{N}\nabla^{\delta}u_{i}^{\varepsilon}\cdot\nabla\theta^{\varepsilon}\varphi dx,}\\
	\\
	{\displaystyle \int_{\Omega^{\varepsilon}}\partial_{t}u_{i}^{\varepsilon}\phi_{i}dx+\int_{\Omega^{\varepsilon}}d_{i}^{\varepsilon}\nabla u_{i}^{\varepsilon}\cdot\nabla\phi_{i}dx+\varepsilon\int_{\Gamma^{\varepsilon}}\left(a_{i}^{\varepsilon}u_{i}^{\varepsilon}-b_{i}^{\varepsilon}v_{i}^{\varepsilon}\right)\phi_{i}dS_{\varepsilon}}\quad\quad\quad\quad\quad\quad\quad\quad\\
	{\displaystyle \quad\quad\quad\quad\quad\quad\quad\quad\quad\quad\quad\quad\quad\quad\quad\quad=\int_{\Omega^{\varepsilon}}R_{i}\left(u^{\varepsilon}\right)\phi_{i}dx+\int_{\Omega^{\varepsilon}}\rho_{i}^{\varepsilon}\nabla^{\delta}\theta^{\varepsilon}\cdot\nabla u_{i}^{\varepsilon}\phi_{i}dx,}\\
	\\
	{\displaystyle\varepsilon \int_{\Gamma^{\varepsilon}}\partial_{t}v_{i}^{\varepsilon}\psi_{i}dS_{\varepsilon}=\varepsilon\int_{\Gamma^{\varepsilon}}\left(a_{i}^{\varepsilon}u_{i}^{\varepsilon}-b_{i}^{\varepsilon}v_{i}^{\varepsilon}\right)\psi_{i}dS_{\varepsilon}},\quad\quad\quad\quad\quad\quad\quad\quad\quad\quad\quad\quad\quad\quad\quad\quad
	\end{array}\right.\label{eq:weak1}
	\end{equation}
	for all $\left(\varphi,\phi_{i},\psi_{i}\right)\in H^{1}\left(\Omega^{\varepsilon}\right)\times H^{1}\left(\Omega^{\varepsilon}\right)\times L^{2}\left(\Gamma^{\varepsilon}\right)$.
\end{definition}

\begin{theorem}\label{thm:wellposed}
	\emph{\textbf{Well-posedness and Positivity of solution}}
	
	Assume $\left(\mbox{A}_{1}\right)$-$\left(\mbox{A}_{2}\right)$ and $i\in\left\{ 1,...,N\right\} $. The microscopic problem $\left(P^{\varepsilon}\right)$ admits a unique solution $\left(\theta^{\varepsilon},u_{i}^{\varepsilon},v_{i}^{\varepsilon}\right)$ in the sense of Definition \ref{def:mathbb=00007BPeps=00007D}, belonging to \[
	K(T,M):=\left\{z\in L^2((0,T)\times \Omega^{\varepsilon}):|z|\le M \;\text{a.e. in}\; (0,T)\times\Omega^{\varepsilon}\right\}\] for some $M>0$. Additionally,
	\[
	\theta^{\varepsilon},u_{i}^{\varepsilon}\in
	H^{1}\left(0,T;L^{2}\left(\Omega^{\varepsilon}\right)\right)\cap
	L^{\infty}\left(0,T;H^{1}\left(\Omega^{\varepsilon}\right)\right)\cap L^{\infty}\left(\left(0,T\right)\times\Omega^{\varepsilon}\right),
	\]
	\[
	v_{i}^{\varepsilon}\in
	H^{1}\left(0,T;L^{2}\left(\Gamma^{\varepsilon}\right)\right)\cap L^{\infty}\left(\left(0,T\right)\times\Gamma^{\varepsilon}\right).
	\]
	Furthermore, this triplet $\left(\theta^{\varepsilon},u_{i}^{\varepsilon},v_{i}^{\varepsilon}\right)$ is positive and the following
	energy estimates hold
	\[
	\kappa_{\min}\left\Vert \nabla\theta^{\varepsilon}\left(t\right)\right\Vert _{L^{2}\left(\Omega^{\varepsilon}\right)}^{2}+\int_{0}^{t}\left\Vert \partial_{t}\theta^{\varepsilon}\left(t\right)\right\Vert _{L^{2}\left(\Omega^{\varepsilon}\right)}^{2}dt\le C,
	\]
	\[
	\left\Vert \nabla u_{i}^{\varepsilon}\left(t\right)\right\Vert _{L^{2}\left(\Omega^{\varepsilon}\right)}^{2}+\int_{0}^{T}\left(\left\Vert \partial_{t}u_{i}^{\varepsilon}\left(t\right)\right\Vert _{L^{2}\left(\Omega^{\varepsilon}\right)}^{2}+\left\Vert \partial_{t}v_{i}^{\varepsilon}\left(t\right)\right\Vert _{L^{2}\left(\Gamma^{\varepsilon}\right)}^{2}\right)dt\le C\quad\mbox{for a.e.}\;t\in\left(0,T\right].
	\]
	
\end{theorem}

We denote
by $\left(P^{0}\right)$ the strong formulation of the macroscopic
(limit) problem. We introduce below the limit problem whose precise structure has been obtained via a two-scale convergence procedure in \cite{KAM14}.

\begin{theorem}\label{thm:P0}
	\emph{\textbf{Strong formulation of the macroscopic problem -- $(P^{0})$}}
	
	Assume $\left(\mbox{A}_{1}\right)$-$\left(\mbox{A}_{2}\right)$. For $i\in\left\{ 1,...,N\right\} $, 
	the triplet $\left(\theta^{0},u_{i}^{0},v_{i}^{0}\right)$ of limit solutions $\left(\theta^{\varepsilon},u_{i}^{\varepsilon},v_{i}^{\varepsilon}\right)$ to $\left(P^{\varepsilon}\right)$ in the sense of Definition \ref{def:mathbb=00007BPeps=00007D} satisfies the following macroscopic system
	\begin{equation}
	{\displaystyle \partial_{t}\theta^{0}+\nabla\cdot\left(-\mathbb{K}\nabla\theta^{0}\right)+g_{0}\frac{\left|\Gamma_{R}\right|}{\left|Y_{1}\right|}\theta^{0}=\sum_{i=1}^{N}\left(\mathbb{T}^{i}\nabla^{\delta}u_{i}^{0}\right)\cdot\nabla\theta^{0}}\quad\mbox{in}\;Q_{T},
	\label{eq:strongmac1}
	\end{equation}
	\begin{equation}
	{\displaystyle \partial_{t}u_{i}^{0}+\nabla\cdot\left(-\mathbb{D}^{i}\nabla u_{i}^{0}\right)+A_{i}u_{i}^{0}-B_{i}v_{i}^{0}}=\left(\mathbb{F}^{i}\nabla u_{i}^{0}\right)\cdot\nabla^{\delta}\theta^{0}+R_{i}\left(u^{0}\right)\quad\mbox{in}\;Q_{T},
	\label{eq:strongmac2}
	\end{equation}
	subject to the boundary conditions
	\begin{equation}
	-\mathbb{K}\nabla\theta^{0}\cdot\mbox{n}=0\quad\mbox{on}\;\left(0,T\right)\times\partial\Omega,
	\label{eq:strongmac3}
	\end{equation}
	\begin{equation}
	-\mathbb{D}^{i}\nabla u_{i}^{0}\cdot\mbox{n}=0\quad\mbox{on}\;\left(0,T\right)\times\partial\Omega,
	\label{eq:strongmac4}
	\end{equation}
	and associated with the ordinary differential equations
	\begin{equation}
	{\displaystyle \partial_{t}v_{i}^{0}=A_{i}u_{i}^{0}-B_{i}v_{i}^{0}}\quad\mbox{in}\;Q_{T},\label{eq:2.14}
	\end{equation}
	where we have denoted by $\mathbb{K}=K_{0}\mathbb{I}+\left(K_{ij}\right)_{ij}$,
	$\mathbb{T}^{i}=T_{0}^{i}\mathbb{I}+\left(T_{jk}^{i}\right)_{jk}$, $\mathbb{D}^{i}=D_{i}\mathbb{I}+\mathbb{D}_{0}^{i}$,
	$\mathbb{F}^{i}=F_{i}\mathbb{I}+\mathbb{F}_{0}^{i}$ for $j,k\in \left\{1,...,d\right\}$ with $\mathbb{I}$
	standing for the identity matrix and the quantities $K_0, K_{ij}, T_0^i, T_{jk}^{i}, D_{i}, \mathbb{D}_{0}^{i}, F_{i}, \mathbb{F}^{i}, A_{i}, B_{i}$ being effective constants corresponding, respectively, to the oscillating
	coefficients and defined in \eqref{matrix1}-\eqref{matrix5}. 
	
	Furthermore, the initial conditions are provided by
	\begin{equation}
	\theta^{0}\left(t=0\right)=\theta^{0,0}\quad\mbox{in}\;\overline{\Omega},
	\label{eq:strongmac5}
	\end{equation}
	\begin{equation}
	u_{i}^{0}\left(t=0\right)=u_{i}^{0,0}\quad\mbox{in}\;\overline{\Omega},
	\label{eq:strongmac6}
	\end{equation}
	\begin{equation}
	v_{i}^{0}\left(t=0\right)=v_{i}^{0,0}\quad\mbox{on}\;\Gamma.
	\label{eq:strongmac7}
	\end{equation}
	
\end{theorem}


\begin{theorem}\label{thm:weakmacro}
	\emph{\textbf{The weak formulation of $(P^{0})$}}
	
	Assume $\left(\mbox{A}_{1}\right)$-$\left(\mbox{A}_{2}\right)$ and take $i\in\left\{ 1,...,N\right\} $, the triplet $\left(\theta^{0},u_{i}^{0},v_{i}^{0}\right)$
	satifying
	\[
	\theta^{0},u_{i}^{0}\in
	H^{1}\left(0,T;L^{2}\left(\Omega\right)\right)\cap
	L^{2}\left(0,T;H^{1}\left(\Omega\right)\right)\cap L^{\infty}\left(\left(0,T\right)\times\Omega\right),
	\]
	\[
	v_{i}^{0}\in H^{1}\left(0,T;L^{2}\left(\Omega\right)\right)\cap L^{\infty}\left(\left(0,T\right)\times\Omega\right),
	\]
	is a weak solution to $(P^0)$ provided that
	\begin{equation}
	\left\{ \begin{array}{c}
	{\displaystyle \int_{\Omega}\partial_{t}\theta^{0}\varphi dx+\int_{\Omega}\mathbb{K}\nabla\theta^{0}\cdot\nabla\varphi dx+g_{0}\frac{\left|\Gamma_{R}\right|}{\left|Y_{1}\right|}\int_{\Omega}\theta^{0}\varphi dx=\int_{\Omega}\sum_{i=1}^{N}\left(\mathbb{T}^{i}\nabla^{\delta}u_{i}^{0}\right)\cdot\nabla\theta^{0}\varphi dx,}\\
	\\
	{\displaystyle \int_{\Omega}\partial_{t}u_{i}^{0}\phi_{i}dx+\int_{\Omega}\mathbb{D}^{i}\nabla u_{i}^{0}\cdot\nabla\phi_{i}dx+\int_{\Omega}\left(A_{i}u_{i}^{0}-B_{i}v_{i}^{0}\right)\phi_{i}dx}\quad\quad\quad\quad\quad\quad\quad\quad\quad\\
	{\displaystyle \quad\quad\quad\quad\quad\quad\quad\quad\quad\quad\quad\quad\quad\quad\quad=\int_{\Omega}\left(\mathbb{F}^{i}\nabla u_{i}^{0}\right)\cdot\nabla^{\delta}\theta^{0}\phi_{i}dx+\int_{\Omega}R_{i}\left(u^{0}\right)\phi_{i}dx,}\\
	\\
	{\displaystyle \int_{\Omega}\partial_{t}v_{i}^{0}\psi_{i}dx=\int_{\Omega}\left(A_{i}u_{i}^{0}-B_{i}v_{i}^{0}\right)\psi_{i}dx},\quad\quad\quad\quad\quad\quad\quad\quad\quad\quad\quad\quad\quad\quad\quad\quad\quad
	\end{array}\right.\label{eq:weak2}
	\end{equation}
	hold for all $\left(\varphi,\phi_{i},\psi_{i}\right)\in C^{\infty}\left(\Omega\right)\times C^{\infty}\left(\Omega\right)\times C^{\infty}\left(\Omega\right)$.
\end{theorem}


For $i\in \left\{1,...,N\right\}$ and $j,k\in \left\{1,...,d\right\}$, the effective constants in Theorem \ref{thm:P0} are defined, as follows:
\begin{equation}
K_{0}:=\frac{1}{\left|Y_{1}\right|}\int_{Y_{1}}\kappa\left(y\right)dy,\quad K_{ij}:=\frac{1}{\left|Y_{1}\right|}\int_{Y_{1}}\kappa\left(y\right)\frac{\partial\bar{\theta}^{j}}{\partial y_{i}}dy,
\label{matrix1}
\end{equation}
\begin{equation}
T_{0}^{i}:=\frac{1}{\left|Y_{1}\right|}\int_{Y_{1}}\tau_{i}\left(y\right)dy,\quad T_{jk}^{i}:=\frac{1}{\left|Y_{1}\right|}\int_{Y_{1}}\tau_{i}\left(y\right)\frac{\partial\bar{\theta}^{j}}{\partial y_{i}}dy,
\label{matrix2}
\end{equation}
\begin{equation}
D_{i}:=\frac{1}{\left|Y_{1}\right|}\int_{Y_{1}}d_{i}\left(y\right)dy,\quad\mathbb{D}_{0}^{i}:=\left(\frac{1}{\left|Y_{1}\right|}\int_{Y_{1}}d_{i}\left(y\right)\frac{\partial\bar{u}_{i}^{j}}{\partial y_{k}}dy\right)_{jk},
\label{matrix3}
\end{equation}
\begin{equation}
F_{i}:=\frac{1}{\left|Y_{1}\right|}\int_{Y_{1}}\rho_{i}\left(y\right)dy,\quad\mathbb{F}^{i}:=\left(\frac{1}{\left|Y_{1}\right|}\int_{Y_{1}}\rho_{i}\left(y\right)\frac{\partial\bar{u}_{i}^{j}}{\partial y_{k}}dy\right)_{jk},
\label{matrix4}
\end{equation}
\begin{equation}
A_{i}:=\frac{1}{\left|Y_{1}\right|}\int_{\partial Y_0}a_{i}dy,\quad B_{i}:=\frac{1}{\left|Y_{1}\right|}\int_{\partial Y_0}b_{i}dy.
\label{matrix5}
\end{equation}

Hereby, the functions $\bar{\theta}$ and $\bar{u}_{i}$ linearly formulate
the limit functions $\theta^{1}$ and $u_{i}^{1}$ by $\theta^{1}:=\bar{\theta}\cdot\nabla_{x}\theta^{0}=\displaystyle{\sum_{j=1}^{d}}\partial_{x_j}\theta^{0} \bar{\theta}^{j}$
and $u_{i}^{1}:=\bar{u}_{i}\cdot\nabla_{x} u_{i}^{0}=\displaystyle{\sum_{j=1}^{d}}\partial_{x_j}u_i^0 \bar{u}_{i}^{j}$ for $i\in \left\{1,...,N\right\}$. Moreover, they solve, respectively, the cell problems introduced in the following Theorem.

\begin{theorem}\label{thm:cell}
	\emph{\textbf{The cell problems}}
	
	Assume $\left(\mbox{A}_{1}\right)$ holds. The limit functions $\theta^{1}$ and $u_{i}^{1}$ defined as above 
	solve the following cell problems:
	\begin{equation}
	\left\{ \begin{array}{c}
	\nabla_{y}\cdot\left(-\kappa\left(y\right)\nabla_{y}\bar{\theta}^{j}(x,y)\right)=\nabla_{y}\cdot(\kappa n_j)\quad\mbox{in}\;Y_{1},\\
	\\
	-\kappa\left(y\right)\nabla_{y}\bar{\theta}^{j}\cdot\mbox{n}=\kappa n_{j}\quad\mbox{on}\;\partial Y_{0},\\
	\\
	\bar{\theta}^{j}\;\mbox{is}\;Y\mbox{-periodic},
	\end{array}\right.\label{eq:cell1}
	\end{equation}
	\begin{equation}
	\left\{ \begin{array}{c}
	\nabla_{y}\cdot\left(-d_{i}\left(y\right)\nabla_{y}\bar{u}_{i}^{j}(x,y)\right)=\nabla_{y}\cdot (d_i n_j)\quad\mbox{in}\;Y_{1},\\
	\\
	-d_{i}\left(y\right)\nabla_{y}\bar{u}_{i}^{j}\cdot\mbox{n}=d_{i}n_{j}\quad\mbox{on}\;\partial Y_{0},\\
	\\
	\bar{u}_{i}^{j}\;\mbox{is}\;Y\mbox{-periodic},
	\end{array}\right.\label{eq:cell2}
	\end{equation}
	where $n_j$ is the $j$th unit vector of $\mathbb{R}^d$ and $i\in \left\{1,...,N\right\},j\in \left\{1,...,d\right\}$.
	Furthermore,
	\begin{itemize}
		\item[(i)] If $\kappa,d_{i}\in [H^{1}\left(\bar{Y}_{1}\right)]^{d^2}$ are Lipschitz continuous, the system \eqref{eq:cell1}-\eqref{eq:cell2} 
		admits a unique solution $\left(\bar{\theta}^{j},\bar{u}_{i}^{j}\right)\in H^{2}_{loc}\left(Y_{1}\right)\times H^{2}_{loc}\left(Y_{1}\right)$;
		\item[(ii)] If $k,d_{i}\in [H^{1}\left(Y_{1}\right)]^{d^2}\cap [H^{-\frac{1}{2}+s}\left(\partial Y_{0}\right)]^{d^2}$
		for every $s\in\left(-\frac{1}{2},\frac{1}{2}\right)$ are Lipschitz continuous, the system  \eqref{eq:cell1}-\eqref{eq:cell2}  admits a unique
		solution $\left(\bar{\theta}^{j},\bar{u}_{i}^{j}\right)\in H^{1+s}\left(Y_{1}\right)\times H^{1+s}\left(Y_{1}\right)$.
	\end{itemize}
\end{theorem}

The weak solvability of the cell problems \eqref{eq:cell1} and \eqref{eq:cell2} shall be further discussed in the proof of our main result -- Theorem \ref{mainthm:1}. To derive our corrector estimates, we need a number of elementary inequalities.
\begin{itemize}
	\item For all $1\le p\le\infty$,
	the following estimates hold:
	\begin{equation}
	\left\Vert \nabla^{\delta}f\cdot g\right\Vert _{L^{p}\left(\Omega^{\varepsilon}\right)}\le C_{\delta}\left\Vert f\right\Vert _{L^{\infty}\left(\Omega^{\varepsilon}\right)}\left\Vert g\right\Vert _{\left[L^{p}\left(\Omega^{\varepsilon}\right)\right]^{d}}\quad\mbox{for}\;f\in L^{\infty}\left(\Omega^{\varepsilon}\right),g\in\left[L^{p}\left(\Omega^{\varepsilon}\right)\right]^{d},
	\label{thm:moliinequal-2}
	\end{equation}
	\begin{equation}
	\left\Vert \nabla^{\delta}f\right\Vert _{L^{p}\left(\Omega^{\varepsilon}\right)}\le C_{\delta}\left\Vert f\right\Vert _{L^{2}\left(\Omega^{\varepsilon}\right)}\quad\mbox{for}\;f\in L^{2}\left(\Omega^{\varepsilon}\right),
	\label{thm:moliinequal}
	\end{equation}
	where $C>0$ depends only on $\delta$. See \cite{KAM14}, e.g., for a proof of \eqref{thm:moliinequal-2} and \eqref{thm:moliinequal}.
	
	\item To estimate the correctors for both the temperature $\theta^{\varepsilon}$ and colloidal concentrations $u_i^{\varepsilon}$, we consider the real-valued cut-off function $m^{\varepsilon}\in C_{0}^{1}\left(\Omega\right)$
	satisfying $0\le m^{\varepsilon}\le1$, $\varepsilon\left|\nabla m^{\varepsilon}\right|\le C$,
	and $m^{\varepsilon}=1$ on $\left\{ x\in\Omega:\mbox{dist}\left(x,\Gamma\right)\ge\varepsilon\right\} $.
	Furthermore, one can prove that
	\begin{equation}
	\left\Vert 1-m^{\varepsilon}\right\Vert _{L^{2}\left(\Omega^{\varepsilon}\right)}\le C\varepsilon^{1/2},\quad\varepsilon\left\Vert \nabla m^{\varepsilon}\right\Vert _{L^{2}\left(\Omega^{\varepsilon}\right)}\le C\varepsilon^{1/2}.
	\label{rem:cut-off}
	\end{equation}
	
	\item (A Young-type inequality) Let $\delta>0$ and $a,b\ge0$ be arbitrarily
	real numbers and take $q,q'>1$ real constants that are H\"older
	conjugates of each other. Then the following inequality holds
	\begin{equation}
	ab\le\frac{1}{q}\delta^{q}a^{q}+\frac{1}{q'}\delta^{-q'}b^{q'}.
	\label{eq:young}
	\end{equation}
	
	\item (Trace inequality for $\varepsilon$-dependent hypersurfaces $\Gamma^{\varepsilon}$) Let $\Gamma^{\varepsilon}$ be as in Subsection \ref{subsec:geometry}. For $\varphi^{\varepsilon}\in H^1 (\Omega^{\varepsilon})$, there exists a constant $C>0$ (independent of $\varepsilon$) such that
	\begin{equation}
	\varepsilon\left\Vert \varphi^{\varepsilon}\right\Vert _{L^{2}\left(\Gamma^{\varepsilon}\right)}^{2}\le C\left(\left\Vert \varphi^{\varepsilon}\right\Vert _{L^{2}\left(\Omega^{\varepsilon}\right)}^{2}+\varepsilon^{2}\left\Vert \nabla\varphi^{\varepsilon}\right\Vert _{L^{2}\left(\Omega^{\varepsilon}\right)}^{2}\right).
	\label{eq:tracein}
	\end{equation}
	The proof of \eqref{eq:tracein} can be found in \cite[Lemma 3]{HJ91}.
\end{itemize}

\begin{theorem}
	\emph{\textbf{Existence and uniqueness results for $(P^{0})$}}
	Assume $\left(\mbox{A}_{1}\right)$-$\left(\mbox{A}_{2}\right)$.
	For $i\in\left\{ 1,...,N\right\} $, the macroscopic problem $\left(P^{0}\right)$ admits a unique
	(local) weak solution in $L^{2}\left(\left(0,T\right)\times\Omega\right)$.
\end{theorem}

\begin{proof}
	Due to the homogenization limit results in \cite[Lemma 4.3]{KAM14}, the existence of the triplet $\left(\theta^{0},u_{i}^{0},v_{i}^{0}\right)$ in Theorem \ref{thm:weakmacro} is guaranteed. The contraction of these functions in a closed subspace of  $[L^2((0,T)\times \Omega)]^{N+2}$ can be proved concisely by a linearization argument. The proof can be sketched as follows: We define
	\[
	K_{1}\left(M,T\right):=\left\{ z\in L^{2}\left(\left(0,T\right)\times\Omega\right):\left|z\right|\le M\;\text{a.e. in}\;Q_{T}\right\} .
	\]
	For $i\in\left\{ 1,...,N\right\} $, let $\theta^{0,1},u_{i}^{0,1},v_{i}^{0,1}\in K_{1}\left(M_{1},T_{1}\right)$
	and $\theta^{0,2},u_{i}^{0,2},v_{i}^{0,2}\in K_{1}\left(M_{2},T_{2}\right)$
	be two pairs of (weak) solutions of the macro system. By choosing $T=\min\left\{ T_{1},T_{2}\right\} $
	and $M=2\max\left\{ M_{1},M_{2}\right\} $ and suitable test functions $\varphi,\phi_i,\psi_i$ in \eqref{eq:weak2}, we get $d\left(\theta^{0}\right):=\theta^{0,1}-\theta^{0,2},d\left(u_{i}^{0}\right):=u_{i}^{0,1}-u_{i}^{0,2},d\left(v_{i}^{0}\right):=v_{i}^{0,1}-v_{i}^{0,2}\in K_{1}\left(M,T\right)$,
	which satisfy the following equalities:
	\begin{align}
	\frac{1}{2}\partial_{t}\left\Vert d\left(\theta^{0}\right)\right\Vert _{L^{2}\left(\Omega\right)}^{2}+\mathbb{K}\left\Vert \nabla d\left(\theta^{0}\right)\right\Vert _{L^{2}\left(\Omega\right)}^{2} & +g_{0}\frac{\left|\Gamma_{R}\right|}{\left|Y_{1}\right|}\left\Vert d\left(\theta^{0}\right)\right\Vert _{L^{2}\left(\Omega\right)}^{2}\nonumber \\
	= & \int_{\Omega}\sum_{i=1}^{N}\left(\left(\mathbb{T}^{i}\nabla^{\delta}u_{i}^{0,1}\right)\cdot\nabla\theta^{0,1}-\left(\mathbb{T}^{i}\nabla^{\delta}u_{i}^{0,2}\right)\cdot\nabla\theta^{0,2}\right)d\left(\theta^{0}\right)dx,\label{eq:0.1}
	\end{align}
	\begin{align}
	\frac{1}{2}\partial_{t}\left\Vert d\left(u_{i}^{0}\right)\right\Vert _{L^{2}\left(\Omega\right)}^{2}+\mathbb{D}^{i}\left\Vert \nabla d\left(u_{i}^{0}\right)\right\Vert _{L^{2}\left(\Omega\right)}^{2} & +A_{i}\left\Vert d\left(u_{i}^{0}\right)\right\Vert _{L^{2}\left(\Omega\right)}^{2}-\int_{\Omega}B_{i}d\left(v_{i}^{0}\right)d\left(u_{i}^{0}\right)dx\nonumber \\
	& =\int_{\Omega}\left(\left(\mathbb{F}^{i}\nabla u_{i}^{0,1}\right)\cdot\nabla^{\delta}\theta^{0,1}-\left(\mathbb{F}^{i}\nabla u_{i}^{0,2}\right)\cdot\nabla^{\delta}\theta^{0,2}\right)d\left(u_{i}^{0}\right)dx\nonumber \\
	& +\int_{\Omega}\left(R_{i}\left(u_{i}^{0,1}\right)-R_{i}\left(u_{i}^{0,2}\right)\right)d\left(u_{i}^{0}\right)dx,\label{eq:0.2}
	\end{align}
	\[
	\frac{1}{2}\partial_{t}\left\Vert d\left(v_{i}^{0}\right)\right\Vert _{L^{2}\left(\Omega\right)}^{2}+B_{i}\left\Vert d\left(v_{i}^{0}\right)\right\Vert _{L^{2}\left(\Omega\right)}^{2}=\int_{\Omega}A_{i}d\left(u_{i}^{0}\right)d\left(v_{i}^{0}\right)dx.
	\]
	Then, with the help of the estimates \eqref{thm:moliinequal-2}-\eqref{thm:moliinequal} and the Young-type inequality \eqref{eq:young} under
	a suitable choice of a pair $\left(\delta,q,q'\right)$ to get rid
	of the gradient norms $\left\Vert \nabla d\left(\theta^{0}\right)\right\Vert _{L^{2}\left(\Omega\right)}^{2}$
	and $\left\Vert \nabla d\left(u_{i}^{0}\right)\right\Vert _{L^{2}\left(\Omega\right)}^{2}$
	on the left-hand side of (\ref{eq:0.1})-(\ref{eq:0.2}), one can
	find a constant $C\left(M.\delta\right)>0$ such that for all $i\in\left\{ 1,...,N\right\} $
	\begin{align}
	\partial_{t}\left\Vert d\left(\theta^{0}\right)\right\Vert _{L^{2}\left(\Omega\right)}^{2} & +\partial_{t}\left\Vert d\left(u_{i}^{0}\right)\right\Vert _{L^{2}\left(\Omega\right)}^{2} +\partial_{t}\left\Vert d\left(v_{i}^{0}\right)\right\Vert _{L^{2}\left(\Omega\right)}^{2}\nonumber \\
	&\le C\left(M,\delta\right)\left(\left\Vert d\left(\theta^{0}\right)\right\Vert _{L^{2}\left(\Omega\right)}^{2}+\left\Vert d\left(u_{i}^{0}\right)\right\Vert _{L^{2}\left(\Omega\right)}^{2}+\left\Vert d\left(v_{i}^{0}\right)\right\Vert _{L^{2}\left(\Omega\right)}^{2}+1\right).\label{eq:0.3}
	\end{align}
	
	Hereby, we apply the Gronwall inequality to (\ref{eq:0.3}) and then
	integrate the resulting estimate over $\left(0,T\right)$ to obtain
	that
	\begin{equation}
	\left\Vert d\left(\theta^{0}\right)\right\Vert _{L^{2}\left(\left(0,T\right)\times\Omega\right)}^{2}+\left\Vert d\left(u_{i}^{0}\right)\right\Vert _{L^{2}\left(\left(0,T\right)\times\Omega\right)}^{2}+\left\Vert d\left(v_{i}^{0}\right)\right\Vert _{L^{2}\left(\left(0,T\right)\times\Omega\right)}^{2}\le T^{2}C\left(M,\delta\right)\text{exp}\left(TC\left(M,\delta\right)\right).\label{eq:0.4}
	\end{equation}
	
	Since $T^{2}C\left(M,\delta\right)\text{exp}\left(TC\left(M,\delta\right)\right)\to0$
	as $T\to0$, we can construct an approximation scheme $\left(\theta^{0,n},u_{i}^{0,n},v_{i}^{0,n}\right)$
	for $n\in\mathbb{N}$ for the macro system in which the involved nonlinear
	terms are linearized. With a small enough $T_{0}$ such that $T_{0}^{2}C\left(M,\delta\right)\text{exp}\left(T_{0}C\left(M,\delta\right)\right)<1$,
	we claim that $\left\{ \theta^{0,n}\right\} _{n\in\mathbb{N}},$ $\left\{ u_{i}^{0,n}\right\} _{n\in\mathbb{N}}$ and $\left\{ v_{i}^{0,n}\right\} _{n\in\mathbb{N}}$
	are the Cauchy sequences in $K_{1}\left(M,T_{0}\right)$
	by (\ref{eq:0.4}). Thus, the local existence and uniqueness of solutions in $[L^2((0,T)\times \Omega)]^{N+2}$ 
	to ($P^0$) is guaranteed.
\end{proof}

\section{Main result}\label{sec:3}
The main result of this paper is stated in the next Theorem whose applicability is delimited by the assumptions $\left(\text{A}_{1}\right)$-$\left(\text{A}_{2}\right)$ and the extra regularity assumptions shall also be provided therein. Note that the involved macro reconstructions $\theta^{\varepsilon}_{0},u_{i,0}^{\varepsilon},v_{i,0}^{\varepsilon}$ for $i\in \left\{1,...,N\right\}$ shall be defined right in the next Subsection. 
\begin{theorem}\label{mainthm:1}
	Assume $\left(\text{A}_{1}\right)$-$\left(\text{A}_{2}\right)$. Let $\left(\theta^{\varepsilon},u_{i}^{\varepsilon},v_{i}^{\varepsilon}\right)$
	and $\left(\theta^{0},u_{i}^{0},v_{i}^{0}\right)$ for $i\in\left\{ 1,...,N\right\} $
	be weak solutions to $(P^{\varepsilon})$ and $(P^0)$ in the sense of Definition \ref{def:mathbb=00007BPeps=00007D} and Theorem \ref{thm:weakmacro}, respectively.
	Let $\bar{\theta},\bar{u}_i$ be the cell functions solving the cell problems \eqref{eq:cell1}-\eqref{eq:cell2} and satisfying
	\[
	\bar{\theta},\bar{u}_{i}\in L^{\infty}\left(\Omega^{\varepsilon};W_{\#}^{1+s,2}\left(Y_{1}\right)\right)\cap H^{1}\left(\Omega^{\varepsilon};W_{\#}^{s,2}\left(Y_{1}\right)\right)\quad\text{for}\;s>d/2.
	\]
	For every $t\in (0,T]$, we also assume that
	$\theta^{0}\left(t,\cdot\right),u_{i}^{0}\left(t,\cdot\right)\in W^{1,\infty}\left(\Omega^{\varepsilon}\right)\cap H^{2}\left(\Omega^{\varepsilon}\right)$ for $i\in\left\{ 1,...,N\right\} $. On top of that, we assume the initial homogenization limit is of the rate
	\[
	\left\Vert \theta^{\varepsilon,0}-\theta^{0,0}\right\Vert _{L^{2}\left(\Omega^{\varepsilon}\right)}^{2}+\sum_{i=1}^{N}\left\Vert u_{i}^{\varepsilon,0}-u_{i}^{0,0}\right\Vert _{L^{2}\left(\Omega^{\varepsilon}\right)}^{2}+\sum_{i=1}^{N}\left\Vert v_{i}^{\varepsilon,0}-v_{i}^{0,0}\right\Vert _{L^{2}\left(\Gamma^{\varepsilon}\right)}^{2}\le\varepsilon^{\gamma},
	\]
	for some $\gamma\in\mathbb{R}_{+}$. Then the following corrector estimate holds\\
	
	$\displaystyle{
		\left\Vert \theta^{\varepsilon}-\theta^{0}\right\Vert _{L^{2}\left((0,T)\times\Omega^{\varepsilon}\right)}^{2}+\sum_{i=1}^{N}\left\Vert u_{i}^{\varepsilon}-u_{i}^{0}\right\Vert _{L^{2}\left((0,T)\times\Omega^{\varepsilon}\right)}^{2}
	}
	$
	
	$\displaystyle{
		+\left\Vert \nabla\left(\theta^{\varepsilon}-\theta_{1}^{\varepsilon}\right)\right\Vert _{L^{2}\left(0,T;\left[L^{2}\left(\Omega^{\varepsilon}\right)\right]^{d}\right)}^{2}
		+\sum_{i=1}^{N}\left\Vert \nabla\left(u_{i}^{\varepsilon}-u_{i,1}^{\varepsilon}\right)\right\Vert _{L^{2}\left(0,T;\left[L^{2}\left(\Omega^{\varepsilon}\right)\right]^{d}\right)}^{2}\le C\max\left\{ \varepsilon,\varepsilon^{\gamma}\right\},	
	}	
	$\\	
	where $C$ is a generic positive constant that is independent of $\varepsilon$.
	
	Furthermore, if $\gamma \ge 1$, then we obtain
	\[
	\varepsilon\sum_{i=1}^{N}\left\Vert v_{i}^{\varepsilon}-v_{i}^{0}\right\Vert _{L^{2}\left((0,T)\times\Gamma^{\varepsilon}\right)}^{2}\le C \varepsilon.
	\]
\end{theorem}
\subsection{Macroscopic reconstruction}

To derive correctors estimates for our problem, we use the concept of the macroscopic reconstruction. We borrow this terminology from Eck\cite{Eck2}, but note that it is also connected to similar concepts in the \emph{a posteriori} numerical analysis of PDEs (see e.g. \cite{GLV11}). It turns
out that we derive operators that could bring us the link between
the strong formulations $\left(P^{\varepsilon}\right)$ and $\left(P^{0}\right)$. For
a.e. $t\in\left[0,T\right]$ and $x\in\Omega^{\varepsilon}$ we provide
that 
\begin{equation}
\theta^{\varepsilon}_{0}\left(t,x\right):=\theta^{0}\left(t,x\right),
\label{eq:macrore1}
\end{equation}
\begin{equation}
u_{i,0}^{\varepsilon}\left(t,x\right):=u_{i}^{0}\left(t,x\right),
\label{eq:macrore2}
\end{equation}
\begin{equation}
v_{i,0}^{\varepsilon}\left(t,x\right):=v_{i}^{0}\left(t,x\right).
\label{eq:macrore3}
\end{equation}
Henceforward, we obtain the system of macroscopic reconstruction whose
expression is similar to the strong formulations $\left(P^{0}\right)$,
but acting on $x\in\Omega^{\varepsilon}$. We accordingly subtract
this system from the microscopic system $\left(P^{\varepsilon}\right)$
equation-by-equation and gain the difference system over $\Omega^{\varepsilon}$. Then we proceed to the correctors justification
by the following choice of test functions:
\begin{equation}
\varphi\left(t,x\right):=\theta^{\varepsilon}\left(t,x\right)-\left(\theta^{\varepsilon}_{0}\left(t,x\right)+\varepsilon m^{\varepsilon}\left(x\right)\bar{\theta}\left(x,\frac{x}{\varepsilon}\right)\cdot\nabla_{x}\theta^{0}\left(t,x\right)\right),\label{eq:test1}
\end{equation}
\begin{equation}
\phi_{i}\left(t,x\right):=u_{i}^{\varepsilon}\left(t,x\right)-\left(u_{i,0}^{\varepsilon}\left(t,x\right)+\varepsilon m^{\varepsilon}\left(x\right)\bar{u}_{i}\left(x,\frac{x}{\varepsilon}\right)\cdot\nabla_{x}u_{i}^{0}\left(t,x\right)\right),\label{eq:test2}
\end{equation}
where $m_{\varepsilon}$ is a cut-off function with the properties \eqref{rem:cut-off}.

Multiplying the difference system by the test functions $\varphi,\phi_i\in H^1(\Omega^{\varepsilon})$ and integrating the resulting equations
over $\Omega^{\varepsilon}$, we obtain the system, denoted by $\left(\bar{\mathbb{P}}^{\varepsilon}\right)$, as follows:
\[
\begin{array}{c}
{\displaystyle \int_{\Omega^{\varepsilon}_{0}}\partial_{t}\left(\theta^{\varepsilon}-\theta^{\varepsilon}_{0}\right)\varphi dx+\int_{\Omega^{\varepsilon}}\left(\kappa^{\varepsilon}\nabla\theta^{\varepsilon}-\mathbb{K}\nabla\theta^{\varepsilon}_{0}\right)\cdot\nabla\varphi dx+\varepsilon\int_{\Gamma_{R}^{\varepsilon}}g_{0}\theta^{\varepsilon}\varphi dS_{\varepsilon}\quad\quad\quad\quad\quad\quad\quad}\\
{\displaystyle -g_{0}\frac{\left|\Gamma_{R}\right|}{\left|Y_{1}\right|}\int_{\Omega^{\varepsilon}}\theta^{\varepsilon}_{0}\varphi dx=\int_{\Omega^{\varepsilon}}\left(\tau^{\varepsilon}\sum_{i=1}^{N}\nabla^{\delta}u_{i}^{\varepsilon}\cdot \nabla \theta^{\varepsilon}-\sum_{i=1}^{N}\left(\mathbb{T}^{i}\nabla^{\delta}u_{i,0}^{\varepsilon}\right)\cdot\nabla\theta^{\varepsilon}_{0}\right)\varphi dx,}
\end{array}
\]
${\displaystyle {\displaystyle \int_{\Omega^{\varepsilon}}\partial_{t}\left(u_{i}^{\varepsilon}-u_{i,0}^{\varepsilon}\right)\phi_{i}dx+\int_{\Omega^{\varepsilon}}\left(d_{i}^{\varepsilon}\nabla u_{i}^{\varepsilon}-\mathbb{D}^{i}\nabla u_{i,0}^{\varepsilon}\right)\cdot\nabla\phi_{i}dx}+\varepsilon\int_{\Gamma^{\varepsilon}}\left(a_{i}^{\varepsilon}u_{i}^{\varepsilon}-b_{i}^{\varepsilon}v_{i}^{\varepsilon}\right)\phi_{i}dS_{\varepsilon}}$
\begin{eqnarray*}
	-\int_{\Omega^{\varepsilon}}\left(A_{i}u_{i,0}^{\varepsilon}-B_{i}v_{i,0}^{\varepsilon}\right)\phi_{i}dx & = & \int_{\Omega^{\varepsilon}}\left(\rho_{i}^{\varepsilon}\nabla^{\delta}\theta^{\varepsilon}\cdot\nabla u_{i}^{\varepsilon}-\left(\mathbb{F}^{i}\nabla u_{i,0}^{\varepsilon}\right)\cdot\nabla^{\delta}\theta^{\varepsilon}_{0}\right)\phi_{i}dx\\
	&  & +\int_{\Omega^{\varepsilon}}\left(R_{i}\left(u^{\varepsilon}\right)-R_{i}\left(u_{0}^{\varepsilon}\right)\right)\phi_{i}dx,
\end{eqnarray*}

According to the system $\left(\bar{\mathbb{P}}^{\varepsilon}\right)$, we denote the following terms:
\begin{eqnarray}
\mathcal{I}_{1} & := & \int_{\Omega^{\varepsilon}}\partial_{t}\left(\theta^{\varepsilon}-\theta^{\varepsilon}_{0}\right)\varphi dx, \label{eq:differenceterm1}\\
\mathcal{I}_{2} & := & \int_{\Omega^{\varepsilon}}\left(\kappa^{\varepsilon}\nabla\theta^{\varepsilon}-\mathbb{K}\nabla\theta^{\varepsilon}_{0}\right)\cdot\nabla\varphi dx,\\
\mathcal{I}_{3} & := & \varepsilon\int_{\Gamma_{R}^{\varepsilon}}g_{0}\theta^{\varepsilon}\varphi dS_{\varepsilon}-g_{0}\frac{\left|\Gamma_{R}\right|}{\left|Y_{1}\right|}\int_{\Omega^{\varepsilon}}\theta^{\varepsilon}_{0}\varphi dx,\\
\mathcal{I}_{4} & := & \int_{\Omega^{\varepsilon}}\left(\tau^{\varepsilon}\sum_{i=1}^{N}\nabla^{\delta}u_{i}^{\varepsilon}\cdot\nabla\theta^{\varepsilon}-\sum_{i=1}^{N}\left(\mathbb{T}^{i}\nabla^{\delta}u_{i}^{0}\right)\cdot\nabla\theta^{\varepsilon}_{0}\right)\varphi dx,\\
\mathcal{J}_{1}^{i} & := & \int_{\Omega^{\varepsilon}}\partial_{t}\left(u_{i}^{\varepsilon}-u_{i,0}^{\varepsilon}\right)\phi_{i}dx,\\
\mathcal{J}_{2}^{i} & := & \int_{\Omega^{\varepsilon}}\left(d_{i}^{\varepsilon}\nabla u_{i}^{\varepsilon}-\mathbb{D}^{i}\nabla u_{i,0}^{\varepsilon}\right)\cdot\nabla\phi_{i}dx,\\
\mathcal{J}_{3}^{i} & := & \varepsilon\int_{\Gamma^{\varepsilon}}\left(a_{i}^{\varepsilon}u_{i}^{\varepsilon}-b_{i}^{\varepsilon}v_{i}^{\varepsilon}\right)\phi_{i}dS_{\varepsilon}-\int_{\Omega^{\varepsilon}}\left(A_{i}u_{i,0}^{\varepsilon}-B_{i}v_{i,0}^{\varepsilon}\right)\phi_{i}dx,\\
\mathcal{J}_{4}^{i} & := & \int_{\Omega^{\varepsilon}}\left(\rho_{i}^{\varepsilon}\nabla^{\delta}\theta^{\varepsilon}\cdot\nabla u_{i}^{\varepsilon}-\left(\mathbb{F}^{i}\nabla u_{i,0}^{\varepsilon}\right)\cdot\nabla^{\delta}\theta^{\varepsilon}_{0}\right)\phi_{i}dx+\int_{\Omega^{\varepsilon}}\left(R_{i}\left(u^{\varepsilon}\right)-R_{i}\left(u_{0}^{\varepsilon}\right)\right)\phi_{i}dx.
\label{eq:differenceterm}
\end{eqnarray}

We introduce, in the same spirit as for \eqref{eq:macrore1} and \eqref{eq:macrore2}, another macroscopic reconstruction $\theta_{1}^{\varepsilon}(t,x)$ and $u_{i,1}^{\varepsilon}(t,x)$ defined as follows:
\[
\theta_{1}^{\varepsilon}(t,x) := \theta_{0}^{\varepsilon}(t,x) + \varepsilon\bar{\theta}\left(x,\frac{x}{\varepsilon}\right)\cdot \nabla_{x}\theta^0(t,x),
\]
\[
u_{i,1}^{\varepsilon}(t,x) := u^{\varepsilon}_{i,0}(t,x) + \varepsilon\bar{u}_{i}\left(x,\frac{x}{\varepsilon}\right)\cdot \nabla_{x}u_i^0(t,x), 
\]
where $\bar{\theta}$ and $\bar{u}_{i}$ are the cell functions introduced in Theorem \ref{thm:cell}.

By definition \eqref{eq:macrore1}-\eqref{eq:macrore2}, the macroscopic reconstruction $\theta_{0}^{\varepsilon}(t,x)$ and $u_{i,0}^{\varepsilon}(t,x)$ are interchangeable, respectively, in notation with the limit functions $\theta^{0}(t,x)$ and $u_{i}^{0}(t,x)$ in Theorem \ref{mainthm:1}.

\subsection{Integral estimates}
\begin{remark}\label{rmk:aaa}
	From Lemma \ref{thm:pep}, one can apply directly the $L^{2}$-estimate
	between the space-dependent physical parameters of the microscopic problem (e.g. $\kappa^{\varepsilon}$,
	$\tau^{\varepsilon}$) and their averages, even if the parameters in discussion are actually tensors.
	To this end, these estimates are controlled as $\left\Vert p^{\varepsilon}-\bar{p}\right\Vert _{L^{2}\left(\Omega^{\varepsilon}\right)}\le C\varepsilon^{1/2}$, where $p^{\varepsilon}$ refers to the oscillating coefficient and $\bar{p}$ denotes its average. 
\end{remark}

\begin{lemma}
	\label{thm:pep}Let $Y_1$ as defined in Subsection \ref{subsec:geometry}. Let $p^{\varepsilon}\left(x\right):=p\left(x/\varepsilon\right)$
	belong to $H^{1}\left(\Omega^{\varepsilon}\right)$ satisfying
	\[
	\bar{p}:=\frac{1}{\left|Y_{1}\right|}\int_{Y_{1}}p\left(y\right)dy.
	\]
	Then the following estimate holds
	\[
	\left\Vert p^{\varepsilon}-\bar{p}\right\Vert _{L^{2}\left(\Omega^{\varepsilon}\right)}\le C\varepsilon^{1/2}\left\Vert p^{\varepsilon}\right\Vert _{H^{1}\left(\Omega^{\varepsilon}\right)}.
	\]
\end{lemma}

\begin{proof}
	We consider the periodic geometry described in Figure \ref{fig:1} in Subsection \ref{subsec:geometry}. For a fixed test function
	$\phi\in H^{1}\left(\Omega^{\varepsilon}\right)$, we see that
	\begin{align*}
	\int_{\Omega^{\varepsilon}}\left(p^{\varepsilon}-\bar{p}\right)\phi dx & =\sum_{k\in\mathbb{Z}^{d}}\int_{\varepsilon Y_{1}^{k}}\left(p^{\varepsilon}-\bar{p}\right)\phi dx\\
	& \le C\int_{\varepsilon Y_{1}}\left(p^{\varepsilon}-\bar{p}\right)\phi dx.
	\end{align*}
	
	By changing the variable $x=\varepsilon y$, the relations
	\begin{align*}
	\int_{\varepsilon Y_{1}}p\left(\frac{x}{\varepsilon}\right)\phi\left(x\right)dx & =\varepsilon^{d}\int_{Y_{1}}p\left(y\right)\phi\left(\varepsilon y\right)dy,\\
	\int_{\varepsilon Y_{1}}\int_{Y_{1}}p\left(y\right)\phi\left(x\right)dydx & =\varepsilon^{d}\int_{Y_{1}}\int_{Y_{1}}p\left(y\right)\phi\left(\varepsilon z\right)dydz,
	\end{align*}
	enable us to write:
	\begin{equation}
	\int_{\varepsilon Y_{1}}\left(p^{\varepsilon}-\bar{p}\right)\phi dx=\varepsilon^{d}\left|Y_{1}\right|^{-1}\int_{Y_{1}}\int_{Y_{1}}\left(p\left(y\right)\phi\left(\varepsilon y\right)-p\left(y\right)\phi\left(\varepsilon z\right)\right)dzdy.\label{eq:3.3-1}
	\end{equation}
	
	Thanks to the representation
	\[
	\phi\left(\varepsilon y\right)-\phi\left(\varepsilon z\right)=\varepsilon\int_{0}^{1}\nabla\phi\left(t\varepsilon y+\left(1-t\right)\varepsilon z\right)\cdot\left(y-z\right)dt,
	\]
	with $\xi=ty+\left(1-t\right)z$ and $\eta=y-z$, we note that  (\ref{eq:3.3-1})
	can be bounded from above by
	\begin{equation}
	\left|\int_{\varepsilon Y_{1}}\left(p^{\varepsilon}-\bar{p}\right)\phi dx\right|
	\le\varepsilon^{d+1}\left|Y_{1}\right|^{-1}\left(\int_{Y_{1}}\int_{Y_{2}}\left|\nabla\phi\left(\varepsilon\xi\right)\cdot\eta\right|^{2}d\eta d\xi\right)^{1/2}\left(\int_{Y_{1}}\int_{Y_{1}}\left|p\left(y\right)\right|^{2}dydz\right)^{1/2}.\label{eq:3.4-1}
	\end{equation}
	
	In \eqref{eq:3.4-1}, we have denoted $Y_{2}:=\left\{ y-z:\;\text{for}\;y,z\in Y_{1}\right\} $. Also, (\ref{eq:3.4-1}) leads  to
	\[
	\int_{\Omega^{\varepsilon}}\left(p^{\varepsilon}-\bar{p}\right)\phi dx\le C\varepsilon\left\Vert p^{\varepsilon}\right\Vert _{L^{2}\left(\Omega^{\varepsilon}\right)}\left\Vert \nabla\phi\right\Vert _{L^{2}\left(\Omega^{\varepsilon}\right)},
	\]
	and with $\phi=p^{\varepsilon}-\bar{p}$ and \eqref{eq:young},
	\eqref{eq:3.4-1} becomes $\left\Vert p^{\varepsilon}-\bar{p}\right\Vert _{L^{2}\left(\Omega^{\varepsilon}\right)}^{2}\le C\varepsilon\left(\left\Vert p^{\varepsilon}\right\Vert _{L^{2}\left(\Omega^{\varepsilon}\right)}^{2}+\left\Vert \nabla p^{\varepsilon}\right\Vert _{L^{2}\left(\Omega^{\varepsilon}\right)}^{2}\right)$ and hence,	
	we finally get
	\[
	\left\Vert p^{\varepsilon}-\bar{p}\right\Vert _{L^{2}\left(\Omega^{\varepsilon}\right)}\le C\varepsilon^{1/2}\left\Vert p^{\varepsilon}\right\Vert _{H^{1}\left(\Omega^{\varepsilon}\right)}.
	\]
	
	This completes the proof of the lemma.
\end{proof}

Due to the no-flux boundary condition \eqref{eq:boundaryNeu}, we define the function space
\[
H^{1}\left(\Gamma_{N}^{\varepsilon}\right):=\left\{ v\in H^{1}\left(\Gamma^{\varepsilon}\right)|-\kappa^{\varepsilon}\nabla v^{\varepsilon}\cdot\text{n}=0\;\text{on}\;\Gamma_{N}^{\varepsilon}\right\} ,
\]
which is a closed subspace of $H^1(\Gamma^{\varepsilon})$. This plays a role inside Lemma \ref{thm:comp}.

\begin{lemma}\label{thm:comp}
	Let $\theta^{\varepsilon}\in L^{2}\left(0,T;H^{1}\left(\Gamma^{\varepsilon}_{N}\right)\right)$
	and $\theta^{0}\in L^{2}\left(0,T;H^{1}\left(\Omega^{\varepsilon}\right)\right)$.
	For any
	\[
	f_{1}\in C\left(\left[0,T\right];H_{+}^{1}\left(\Omega^{\varepsilon}\right)\cap L_{+}^{\infty}\left(\Omega^{\varepsilon}\right)\right),
	\]
	\[
	f_{2}\in C\left(\left[0,T\right];H_{+}^{1}\left(\Gamma^{\varepsilon}\right)\cap L_{+}^{\infty}\left(\Gamma^{\varepsilon}\right)\right),
	\]
	suppose that there exists $f_{3}\in C\left[0,T\right]$ such that
	\[
	\int_{\Omega^{\varepsilon}}f_{1}\theta^{0}dx=\int_{\Gamma^{\varepsilon}_{R}}f_{2}\theta^{\varepsilon}dS_{\varepsilon}+\varepsilon f_{3}.
	\]
	Then, it exists a $C>0$ such that
	\[
	\left|\int_{\Omega^{\varepsilon}}f_{1}\theta^{0}\varphi dx-\varepsilon\int_{\Gamma^{\varepsilon}_{R}}\left(f_{2}\theta^{\varepsilon}+\varepsilon f_{3}\right)\varphi dS_{\varepsilon}\right|\le\varepsilon C\left\Vert \varphi\right\Vert _{H^{1}\left(\Omega^{\varepsilon}\right)},
	\]
	for any $\varphi\in H^{1}\left(\Omega^{\varepsilon}\right)$.
\end{lemma}

\begin{proof}
	We adapt Lemma 5.2 from \cite{Tycho} to our context. The proof of the lemma is based on the following auxiliary problem: Given $f_1,f_2,\theta^{\varepsilon},\theta^0$ as above and $\tilde{f}\in C[0,T]$, find $\Psi $ such that
	\begin{equation}
	\begin{cases}
	\Delta_{y}\left.\Psi\left(\cdot,x,y\right)\right|_{y=\frac{x}{\varepsilon}}=f_{1}\theta^{0} & \;\mbox{for}\;x\in\Omega^{\varepsilon},\\
	\nabla_{y}\Psi\left(\cdot,x,y\right)\cdot\mbox{n}=f_{2}\theta^{\varepsilon}+\varepsilon\tilde{f} & \;\mbox{for}\;(x,y)\in\Gamma_{R}^{\varepsilon},\\
	\nabla_{y}\Psi\cdot\mbox{n}=0 & \;\mbox{at}\;\Gamma_{N}^{\varepsilon}.
	\end{cases}\label{eq:auxi111}
	\end{equation}
	
	By \cite[Lemma 2.1]{PPSW93} and also \cite{CD99}, the problem \eqref{eq:auxi111} has a (weak) $Y$-periodic solution 
	\[
	\left.\Psi\left(\cdot,x,y\right)\right|_{y=\frac{x}{\varepsilon}}\in L^{2}\left(0,T;H^{1}\left(\Omega^{\varepsilon}\right)\right)
	\]
	satisfying the integral equality
	\[
	\int_{\Omega^{\varepsilon}}f_{1}\theta^{0}dx=\int_{\Gamma^{\varepsilon}}\left(f_{2}\theta^{\varepsilon}+\varepsilon\tilde{f}\right)dS_{\varepsilon}=\int_{\Gamma_{R}^{\varepsilon}}f_{2}\theta^{\varepsilon}S_{\varepsilon}+\varepsilon f_{3},
	\]
	with $f_{3}$ being  $\left|\Gamma_{R}^{\varepsilon}\right|^{-1}\tilde{f}$. Moreover, that solution is unique up to an additive constant. 
	
	Multiplying the first equation in \eqref{eq:auxi111} by $\varphi\in H^{1}\left(\Omega^{\varepsilon}\right)$
	and then integrating the resulting equation over $\Omega^{\varepsilon}$,
	we arrive at
	\begin{eqnarray*}
		\left|\int_{\Omega^{\varepsilon}}f_{1}\theta^{0}\varphi dx-\varepsilon\int_{\Gamma^{\varepsilon}_{R}}\left(f_{2}\theta^{\varepsilon}+\varepsilon\tilde{f}\right)\varphi dS_{\varepsilon}\right| & = & \left|\int_{\Omega^{\varepsilon}}\Delta_{y}\left.\Psi\left(\cdot,x,y\right)\right|_{y=\frac{x}{\varepsilon}}\varphi dx-\right.\\
		&  & \left.-\varepsilon\int_{\Gamma_{R}^{\varepsilon}}f_{2}\theta^{\varepsilon}\varphi dS_{\varepsilon}-\varepsilon^{2}\int_{\Gamma_{R}^{\varepsilon}}\tilde{f}\varphi dS_{\varepsilon}\right|
	\end{eqnarray*} 
	
	\begin{eqnarray}
	& = & \left|\int_{\Omega^{\varepsilon}}\varepsilon\left(\nabla_{x}\left[\nabla_{y}\left.\Psi\left(\cdot,x,y\right)\right|_{y=\frac{x}{\varepsilon}}\right]-\nabla_{x}\nabla_{y}\left.\Psi\left(\cdot,x,y\right)\right|_{y=\frac{x}{\varepsilon}}\right)\varphi-\right.\\
	&  & \left.-\varepsilon\int_{\Gamma_{R}^{\varepsilon}}f_{2}\theta^{\varepsilon}\varphi dS_{\varepsilon}-\varepsilon^{2}\left|\Gamma_{R}^{\varepsilon}\right|^{-1}\int_{\Gamma_{R}^{\varepsilon}}f_{3}\varphi dS_{\varepsilon}\right|\\
	& = & \left|\varepsilon\int_{\Gamma^{\varepsilon}}\left(\nabla_{y}\left.\Psi\left(\cdot,x,y\right)\right|_{y=\frac{x}{\varepsilon}}\cdot\mbox{n}\varphi dS_{\varepsilon}-\varepsilon\int_{\Omega^{\varepsilon}}\nabla_{y}\left.\Psi\left(\cdot,x,y\right)\right|_{y=\frac{x}{\varepsilon}}\nabla_{x}\varphi dx\right)-\right.\\
	&  & \left.-\varepsilon\int_{\Omega^{\varepsilon}}\nabla_{x}\nabla_{y}\left.\Psi\left(\cdot,x,y\right)\right|_{y=\frac{x}{\varepsilon}}\varphi dx-\varepsilon\int_{\Gamma_{R}^{\varepsilon}}f_{2}\theta^{\varepsilon}\varphi dS_{\varepsilon}-\varepsilon^{2}\left|\Gamma_{R}^{\varepsilon}\right|^{-1}\int_{\Gamma_{R}^{\varepsilon}}f_{3}\varphi dS_{\varepsilon}\right|.
	\label{eq:haho}
	\end{eqnarray}
	
	Since $\Gamma^{\varepsilon} = \Gamma^{\varepsilon}_{R}\cup \Gamma^{\varepsilon}_{N}$, the choice of boundary conditions in \eqref{eq:auxi111} allows  the boundary integrals in \eqref{eq:haho} to disappear. It follows from the triangle inequality and the H\"older
	inequality that
	\begin{eqnarray*}
		\left|\int_{\Omega^{\varepsilon}}f_{1}\theta^{0}\varphi dx-\varepsilon\int_{\Gamma^{\varepsilon}_{R}}\left(f_{2}\theta^{\varepsilon}+\varepsilon\tilde{f}\right)\varphi dS_{\varepsilon}\right| & \le & \varepsilon\left(\left|\int_{\Omega^{\varepsilon}}\nabla_{y}\left.\Psi\left(\cdot,x,y\right)\right|_{y=\frac{x}{\varepsilon}}\nabla_{x}\varphi dx\right|+\right.\\
		&  & \left.+\left|\int_{\Omega^{\varepsilon}}\nabla_{x}\nabla_{y}\left.\Psi\left(\cdot,x,y\right)\right|_{y=\frac{x}{\varepsilon}}\varphi dx\right|\right)\\
		& \le & C\varepsilon\left\Vert \varphi\right\Vert _{H^{1}\left(\Omega^{\varepsilon}\right)}.
	\end{eqnarray*}
	
	This completes the proof of the lemma.
\end{proof}


\section{Proof of Theorem \ref{mainthm:1}}\label{sec:4}

The proof of Theorem \ref{mainthm:1} relies on a fine control of the $\varepsilon$-dependence needed to estimate each term in \eqref{eq:differenceterm1}-\eqref{eq:differenceterm}. At first, the term $\mathcal{I}_{1}$ can be rewritten as:
\begin{eqnarray}
\int_{\Omega^{\varepsilon}}\partial_{t}\left(\theta^{\varepsilon}-\theta^{0}\right)\left(\theta^{\varepsilon}-\theta^{0}-\varepsilon m^{\varepsilon}\bar{\theta}\left(x,\frac{x}{\varepsilon}\right)\cdot\nabla_{x}\theta^{0}\right) & = & \frac{1}{2}\frac{d}{dt}\left\Vert \theta^{\varepsilon}\left(t\right)-\theta^{0}\left(t\right)\right\Vert _{L^{2}\left(\Omega^{\varepsilon}\right)}^{2}\nonumber \\
&  & -\varepsilon\int_{\Omega^{\varepsilon}}\partial_{t}\left(\theta^{\varepsilon}-\theta^{0}\right)m^{\varepsilon}\bar{\theta}\left(x,\frac{x}{\varepsilon}\right)\cdot\nabla_{x}\theta^{0}dx.\nonumber \\
\label{eq:3.3}
\end{eqnarray}

Similarly, we proceed to estimate $\mathcal{J}_{1}^{i}$ as follows:
\begin{eqnarray}
\int_{\Omega^{\varepsilon}}\partial_{t}\left(u_{i}^{\varepsilon}-u_{i}^{0}\right)\left(u_{i}^{\varepsilon}-u_{i}^{0}-\varepsilon m^{\varepsilon}\bar{u}_{i}\left(x,\frac{x}{\varepsilon}\right)\cdot\nabla_{x}u_{i}^{0}\right) & = & \frac{1}{2}\frac{d}{dt}\left\Vert u_{i}^{\varepsilon}\left(t\right)-u_{i}^{0}\left(t\right)\right\Vert _{L^{2}\left(\Omega^{\varepsilon}\right)}^{2}\nonumber \\
&  & -\varepsilon\int_{\Omega^{\varepsilon}}\partial_{t}\left(u_{i}^{\varepsilon}-u_{i}^{0}\right)m^{\varepsilon}\bar{u}_{i}\left(x,\frac{x}{\varepsilon}\right)\cdot\nabla_{x}u_{i}^{0}dx.\nonumber \\
\label{eq:3.4}
\end{eqnarray}

Using the decomposition
\[
\kappa^{\varepsilon}\nabla\theta^{\varepsilon}-\mathbb{K}\nabla\theta^{0}
=\kappa^{\varepsilon}\nabla\left(\theta^{\varepsilon}-\theta_{1}^{\varepsilon}\right)
+\kappa^{\varepsilon}\nabla\theta_{1}^{\varepsilon}-\mathbb{K}\nabla\theta^{0},
\]
the term $\mathcal{I}_{2}$ thus becomes
\begin{equation}
\mathcal{I}_{2} 
=\int_{\Omega^{\varepsilon}}\kappa^{\varepsilon}\nabla\left(\theta^{\varepsilon}-\theta_{1}^{\varepsilon}\right)\cdot\nabla\varphi dx +\int_{\Omega^{\varepsilon}}\left(\kappa^{\varepsilon}\nabla\theta_{1}^{\varepsilon}-\mathbb{K}\nabla\theta^{0}\right)\cdot\nabla\varphi dx.
\label{eq:3.5}
\end{equation}

Concerning the first term on the right-hand side of (\ref{eq:3.5}), we 
get
\[
\int_{\Omega^{\varepsilon}}\kappa^{\varepsilon}\nabla\left(\theta^{\varepsilon}-\theta_{1}^{\varepsilon}\right)\cdot\nabla\varphi dx\ge\frac{\kappa_{\min}}{2}\left\Vert \nabla\left(\theta^{\varepsilon}-\theta_{1}^{\varepsilon}\right)\left(t\right)\right\Vert _{L^{2}\left(\Omega^{\varepsilon}\right)}^{2}-C\varepsilon^{2}\left\Vert \nabla\left(\left(1-m^{\varepsilon}\right)\bar{\theta}^{\varepsilon}\cdot\nabla_{x}\theta^{0}\left(t\right)\right)\right\Vert _{L^{2}\left(\Omega^{\varepsilon}\right)}^{2}.
\]

It is worth pointing out that the cell problems \eqref{eq:cell1} and \eqref{eq:cell2} require more regularity on the heat conductivity $\kappa$ and the diffusion coefficient $d_i$, namely we need $\kappa, d_i \in H^1(\bar{Y_1})$. On the other side, since these cell problems are elliptic problems on a non-convex polygon, it is well-known  that the cell functions $\bar{\theta}$ and $\bar{u}_{i}$ usually do not belong to $H^2(Y_1)$ in $y$ no matter how smooth the right-hand sides of \eqref{eq:cell1} and \eqref{eq:cell2} are (cf. \cite{Grisvard85}). Due to the extra regularity on $\kappa$ and $d_i$ leading to their Lipschitz property in space and due to the Lipschitz boundary of the microstructure, the solutions can be at most in $H^2_{loc}(\bar{Y_1})$ (see, e.g. \cite[Theorem 2.2.2.3]{Grisvard85}). Notably, that result will not change even if the microstructure boundary is very smooth as in this case. We also emphasize that when investigating problems on domains without holes, the cell problems are then considered in the unit cell $Y$ and by the convexity of that cell, one obtains the regularity of the cell functions up to $H^2(Y)$. 

It follows from \cite[Theorem 4]{Giu98} that the cell problems  \eqref{eq:cell1}-\eqref{eq:cell2} admit a unique solution $(\bar{\theta}, \bar{u}_{i})\in H_{\#}^{1+s}(Y_1)\times H_{\#}^{1+r}(Y_1)$ for some $s,r\in (-\frac{1}{2},\frac{1}{2})$. Essentially, this hinders us when dealing with the term $\varepsilon\left\Vert \nabla\left(\left(1-m^{\varepsilon}\right)\bar{\theta}^{\varepsilon}\cdot\nabla_{x}\theta^{0}\left(t\right)\right)\right\Vert _{L^{2}\left(\Omega^{\varepsilon}\right)}$. In fact, we need $\bar{\theta}\in L^{\infty}(\Omega^{\varepsilon};C^1_{\#}(\bar{Y_1}))$, whereas its maximal regularity only gives $L^{\infty}(\Omega^{\varepsilon};H_{\#}^{1+s}(Y_1))$ (a similar situation holds for $\bar{u}_{i}$). Recall the Sobolev embedding $W^{j+s,p}(Y_1)\subset C^j(\bar{Y_1})$ for $sp>d$ (cf. \cite{Adam75}). Our Hilbertian framework, i.e. $p=2,j=1$, requires $s>d/2\ge 1/2$ which leads to the impossibility of getting $C^1_{\#}(\bar{Y_1})$ from $H_{\#}^{1+s}(Y_1)$. Obviously, one of the possibilities is to working with the domain without holes in 1D, i.e. $d=1$ and $s=1$. The fact that $(\bar{\theta}, \bar{u}_{i})\in [L^{\infty}(\Omega^{\varepsilon};W^{1+s,2}_{\#}(Y_1))]^2$ for $s>d/2$ is strictly needed to obtain $(\bar{\theta}, \bar{u}_{i})\in [L^{\infty}(\Omega^{\varepsilon};C^1_{\#}(\bar{Y_1}))]^2$. Then, with the assumption $\theta^0(t,\cdot) \in W^{1,\infty}(\Omega^\varepsilon) \cap H^2(\Omega^{\varepsilon})$ and the extra regularity $\bar{\theta}\in H^1(\Omega^{\varepsilon};W^{s,2}_{\#}(Y_1))$ providing $\bar{\theta}\in H^1(\Omega^{\varepsilon};C_{\#}(\bar{Y_1}))$,  we estimate that
\begin{align*}
\varepsilon\left\Vert \nabla\left(\left(1-m^{\varepsilon}\right)\bar{\theta}^{\varepsilon}\cdot\nabla_{x}\theta^{0}\left(t\right)\right)\right\Vert _{L^{2}\left(\Omega^{\varepsilon}\right)} & \le\varepsilon\left\Vert \nabla m^{\varepsilon}\right\Vert _{L^{2}\left(\Omega^{\varepsilon}\right)}\left\Vert \bar{\theta}\right\Vert _{L^{\infty}\left(\Omega^{\varepsilon};C\left(\bar{Y_{1}}\right)\right)}\left\Vert \theta^{0}\left(t\right)\right\Vert _{W^{1,\infty}\left(\Omega^{\varepsilon}\right)}\\
& +\varepsilon\left\Vert \nabla_{x}\bar{\theta}\right\Vert _{L^{2}\left(\Omega^{\varepsilon};C\left(\bar{Y_{1}}\right)\right)}\left\Vert \theta^{0}\left(t\right)\right\Vert _{W^{1,\infty}\left(\Omega^{\varepsilon}\right)}\\
& +\left\Vert 1-m^{\varepsilon}\right\Vert _{L^{2}\left(\Omega^{\varepsilon}\right)}\left\Vert \nabla_{y}\bar{\theta}\right\Vert _{L^{\infty}\left(\Omega^{\varepsilon};C\left(\bar{Y_{1}}\right)\right)}\left\Vert \theta^{0}\left(t\right)\right\Vert _{W^{1,\infty}\left(\Omega^{\varepsilon}\right)}\\
& +\varepsilon\left\Vert \bar{\theta}\right\Vert _{L^{\infty}\left(\Omega^{\varepsilon};C\left(\bar{Y_{1}}\right)\right)}\left\Vert \theta^{0}\left(t\right)\right\Vert _{H^{2}\left(\Omega^{\varepsilon}\right)}\\
& \le C\left(\varepsilon+\varepsilon^{1/2}\right),
\end{align*}
where we use the inequalities \eqref{rem:cut-off} together with the fact that $\nabla = \nabla_x + \varepsilon^{-1}\nabla_y$.

Observe that
\begin{equation}
\nabla\theta_{1}^{\varepsilon}=\nabla_x\theta^{0}+\left(\nabla_{y}\bar{\theta}\right)^{\varepsilon}\nabla_{x}\theta^{0}+\varepsilon\bar{\theta}^{\varepsilon}\nabla_{x}\nabla\theta^{0}+\varepsilon\left(\nabla_{x}\bar{\theta}\right)^{\varepsilon}\nabla_{x}\theta^{0}.
\label{eq:difference}
\end{equation}

Hence, we get
\begin{align}
\kappa^{\varepsilon}\nabla\theta_{1}^{\varepsilon}-\mathbb{K}\nabla\theta^{0}\nonumber & =\kappa^{\varepsilon}\left(\nabla\theta^{0}+\left(\nabla_{y}\bar{\theta}\right)^{\varepsilon}\nabla_{x}\theta^{0}\right)-\mathbb{K}\nabla\theta^{0}\\ 
& +\kappa^{\varepsilon}\varepsilon\left(\bar{\theta}^{\varepsilon}\nabla_{x}\nabla\theta^{0}+\left(\nabla_{x}\bar{\theta}\right)^{\varepsilon}\nabla_{x}\theta^{0}\right).
\label{eq:288888}
\end{align}

We note that the $L^2$-norm of the second term on the right-hand side of \eqref{eq:288888} is bounded from above by
\begin{align*}
\varepsilon\left\Vert \kappa^{\varepsilon}\left(\bar{\theta}^{\varepsilon}\nabla_{x}\nabla\theta^{0}+\left(\nabla_{x}\bar{\theta}\right)^{\varepsilon}\nabla_{x}\theta^{0}\right)\right\Vert _{L^{2}\left(\Omega^{\varepsilon}\right)} & \le C\varepsilon\left\Vert \bar{\theta}\right\Vert _{L^{\infty}\left(\Omega^{\varepsilon};C\left(\bar{Y_{1}}\right)\right)}\left\Vert \theta^{0}\right\Vert _{H^{2}\left(\Omega^{\varepsilon}\right)}\\
& +C\varepsilon\left\Vert \nabla_{x}\bar{\theta}\right\Vert _{L^{2}\left(\Omega^{\varepsilon};C\left(\bar{Y_{1}}\right)\right)}\left\Vert \theta^{0}\right\Vert _{W^{1,\infty}\left(\Omega^{\varepsilon}\right)}.
\end{align*}

Let us handle now the remaining quantity $\kappa^{\varepsilon}\left(\nabla\theta^{0}+\left(\nabla_{y}\bar{\theta}\right)^{\varepsilon}\nabla_{x}\theta^{0}\right)-\mathbb{K}\nabla\theta^{0}$. In fact, recall that $\mathcal{G}:=\kappa(\mathbb{I}+\nabla_{y}\bar{\theta})-\mathbb{K}$ is divergence-free with respect to $y\in Y_1$ due to the structure of the cell problems in Theorem \ref{thm:cell}. Moreover, we know that its average also vanishes, i.e.
\[
\int_{Y_{1}}\mathcal{G}dy=0,
\]
by virtue of the definition of the homogenized heat conductivity $\mathbb{K}$ in Theorem \ref{thm:P0}.

As a consequence, $\mathcal{G}$ possesses a vector potential $\mathbf{V}$ and this vector potential is skew-symmetric such that $\mathcal{G} = \nabla_{y}\mathbf{V}$. In general, the selection of the vector potential is non-unique. However, we can choose $\mathbf{V}$ to solve the Poisson equation $\Delta_{y}\mathbf{V} = \eta(x,y)\nabla_{y}\mathcal{G}$ for some function $\eta$ just depending on the dimensions. Using this equation together with the periodic boundary conditions at $\partial Y_0$ and the vanishing cell average, we can determine this vector potential $\mathbf{V}$ uniquely. Now, we formulate the quantity $\mathcal{G}^{\varepsilon}\nabla\theta^0 = \kappa^{\varepsilon}\left(\nabla\theta^{0}+\left(\nabla_{y}\bar{\theta}\right)^{\varepsilon}\nabla_{x}\theta^{0}\right)-\mathbb{K}\nabla\theta^{0}$ in terms of this vector potential. Using the relation that $\nabla_y = \varepsilon\nabla - \varepsilon \nabla_{x}$, we have
\begin{equation}
\mathcal{G}^{\varepsilon}\nabla\theta^0 = \varepsilon\nabla\cdot(\mathbf{V}^{\varepsilon}\nabla\theta^0) 
- \varepsilon \mathbf{V}^{\varepsilon}\Delta\theta^0 
- \varepsilon (\nabla_{x}\mathbf{V})^{\varepsilon}\nabla\theta^0.
\label{eq:qqq}
\end{equation}

Due to the skew-symmetry of $\mathbf{V}$ (and also that of $\mathbf{V}^{\varepsilon}$), the first term on the right-hand side of \eqref{eq:qqq} is divergence-free, indicating the boundedness in $L^2(\Omega^{\varepsilon})$ with the order of $\mathcal{O}(\varepsilon)$. In addition, combining $\bar{\theta}\in L^{\infty}(\Omega^{\varepsilon};W^{1+s,2}_{\#}(Y_1)) \cap H^1(\Omega^{\varepsilon};W^{s,2}_{\#}(Y_1))$ with the above Poisson equation $\Delta_{y}\mathbf{V} = \eta(x,y)\nabla_{y}\mathcal{G}$ yields
\[
\left\Vert \mathbf{V}\right\Vert _{W^{1+s,2}\left(Y_{1}\right)}\le C\left\Vert \mathcal{G}\right\Vert _{W^{s,2}\left(Y_{1}\right)}.
\]

By the compact embedding $W^{s,2}(Y_1)\subset C(\bar{Y_1})$ for $s>d/2\ge 1$, we thus get
\[
\mathbf{V}\in L^{\infty}\left(\Omega^{\varepsilon};C_{\#}\left(\bar{Y_{1}}\right)\right)\cap H^{1}\left(\Omega^{\varepsilon};C_{\#}\left(\bar{Y_{1}}\right)\right).
\]

As a consequence, the boundedness in $L^2(\Omega^{\varepsilon})$ of the second and third terms on the right-hand side of \eqref{eq:qqq} is given by
\[
\varepsilon\left\Vert \mathbf{V}^{\varepsilon}\Delta\theta^{0}+(\nabla_{x}\mathbf{V})^{\varepsilon}\nabla\theta^{0}\right\Vert _{L^{2}\left(\Omega^{\varepsilon}\right)}\le\varepsilon\left\Vert \mathbf{V}\right\Vert _{L^{\infty}\left(\Omega^{\varepsilon};C\left(\bar{Y_{1}}\right)\right)}\left\Vert \theta^{0}\right\Vert _{H^{2}\left(\Omega^{\varepsilon}\right)}+\varepsilon\left\Vert \mathbf{V}\right\Vert _{H^{1}\left(\Omega^{\varepsilon};C\left(\bar{Y_{1}}\right)\right)}\left\Vert \theta^{0}\right\Vert _{W^{1,\infty}\left(\Omega^{\varepsilon}\right)}.
\]

Therefore, with the help of the H\"older inequality, we note that
\[
\int_{\Omega^{\varepsilon}}\left(\kappa^{\varepsilon}\nabla\theta_{1}^{\varepsilon}-\mathbb{K}\nabla\theta^{0}\right)\cdot\nabla\varphi dx
\le C\varepsilon,
\]
which completes the estimates for $\mathcal{I}_{2}$.

Consequently, we can write
\begin{equation}
\mathcal{I}_{2}\ge C\left\Vert \nabla\left(\theta^{\varepsilon}-\theta_{1}^{\varepsilon}\right)\left(t\right)\right\Vert _{\left[L^{2}\left(\Omega^{\varepsilon}\right)\right]^{d}}^{2}-C\left(\varepsilon^{2}+\varepsilon\right)
.\label{eq:3.8}
\end{equation}

Similarly, estimating the term $\mathcal{J}_{2}^{i}$ leads to
\begin{equation}
\mathcal{J}_{2}^{i}\ge C\left\Vert \nabla\left(u_{i}^{\varepsilon}-u_{i,1}^{\varepsilon}\right)\left(t\right)\right\Vert _{\left[L^{2}\left(\Omega^{\varepsilon}\right)\right]^{d}}^{2}-C\left(\varepsilon^{2}+\varepsilon\right).
\label{3.9}
\end{equation}

Concerning the estimate of the term $\mathcal{I}_{3}$, we note the following: Thanks to
the compatibility constraint (Theorem \ref{thm:comp}) with the choice $\varphi = \theta^{\varepsilon}-\theta^{0}$, we get that
\begin{align}
\mathcal{I}_{3} & \le C\varepsilon\left\Vert \varphi\right\Vert _{H^{1}\left(\Omega^{\varepsilon}\right)}\nonumber \\
& \le C\varepsilon\left(\left\Vert \theta^{\varepsilon}-\theta^{0}\right\Vert _{L^{2}\left(\Omega^{\varepsilon}\right)}+\left\Vert \nabla\left(\theta^{\varepsilon}-\theta_{1}^{\varepsilon}\right)\right\Vert _{\left[L^{2}\left(\Omega^{\varepsilon}\right)\right]^{d}}+\left\Vert \nabla\left(\theta_{1}^{\varepsilon}-\theta^{0}\right)\right\Vert _{\left[L^{2}\left(\Omega^{\varepsilon}\right)\right]^{d}}\right)\nonumber \\
& \le C\varepsilon\left(\left\Vert \theta^{\varepsilon}-\theta^{0}\right\Vert _{L^{2}\left(\Omega^{\varepsilon}\right)}+\left\Vert \nabla\left(\theta^{\varepsilon}-\theta_{1}^{\varepsilon}\right)\right\Vert _{\left[L^{2}\left(\Omega^{\varepsilon}\right)\right]^{d}}+C(1+\varepsilon)\right),
\label{3.10}
\end{align}
where we use again the difference relation \eqref{eq:difference} and get the following bound from above
\begin{align*}
\left\Vert \nabla\left(\theta_{1}^{\varepsilon}-\theta^{0}\right)\right\Vert _{L^{2}\left(\Omega^{\varepsilon}\right)} & \le\left\Vert \nabla_{y}\bar{\theta}\right\Vert _{L^{\infty}\left(\Omega^{\varepsilon};C\left(\bar{Y_{1}}\right)\right)}\left\Vert \theta^{0}\right\Vert _{W^{1,\infty}\left(\Omega^{\varepsilon}\right)}\\
& +\varepsilon\left(\left\Vert \bar{\theta}\right\Vert _{L^{\infty}\left(\Omega^{\varepsilon};C\left(\bar{Y_{1}}\right)\right)}\left\Vert \theta^{0}\right\Vert _{H^{2}\left(\Omega^{\varepsilon}\right)}+\left\Vert \nabla_{x}\bar{\theta}\right\Vert _{L^{2}\left(\Omega^{\varepsilon};C\left(\bar{Y_{1}}\right)\right)}\left\Vert \theta^{0}\right\Vert _{W^{1,\infty}\left(\Omega^{\varepsilon}\right)}\right).
\end{align*}

Similarly, the term $\mathcal{J}_{3}^{i}$ is bounded from above by
\begin{equation}
\mathcal{J}_{3}^{i}\le C\varepsilon\left(\left\Vert u_{i}^{\varepsilon}-u_{i}^{0}\right\Vert _{L^{2}\left(\Omega^{\varepsilon}\right)}+\left\Vert \nabla\left(u_{i}^{\varepsilon}-u_{i,1}^{\varepsilon}\right)\right\Vert _{\left[L^{2}\left(\Omega^{\varepsilon}\right)\right]^{d}}+C(1+\varepsilon)\right).\label{3.11}
\end{equation}

Note the elementary decomposition:
\begin{eqnarray*}
	\tau^{\varepsilon}\nabla^{\delta}u_{i}^{\varepsilon}\cdot\nabla\theta^{\varepsilon}-\left(\mathbb{T}^{i}\nabla^{\delta}u_{i}^{0}\right)\cdot\nabla\theta^{0} & = & \left(\tau^{\varepsilon}-\mathbb{T}^{i}\right)\nabla^{\delta}u_{i}^{\varepsilon}\cdot\nabla\theta^{\varepsilon}\\
	&  & +\mathbb{T}^{i}\left(\nabla^{\delta}u_{i}^{\varepsilon}-\nabla^{\delta}u_{i}^{0}\right)\cdot\nabla\theta^{\varepsilon}+\mathbb{T}^{i}\left(\nabla\theta^{\varepsilon}-\nabla\theta^{0}\right)\cdot\nabla^{\delta}u_{i}^{0}.
\end{eqnarray*}

Multiplying the above equation by the test function $\varphi$, we arrive at
\begin{eqnarray*}
	\left(\tau^{\varepsilon}\nabla^{\delta}u_{i}^{\varepsilon}\cdot\nabla\theta^{\varepsilon}-\left(\mathbb{T}^{i}\nabla^{\delta}u_{i}^{0}\right)\cdot\nabla\theta^{0}\right)\varphi & = & \left(\tau^{\varepsilon}-\mathbb{T}^{i}\right)\nabla^{\delta}u_{i}^{\varepsilon}\cdot\nabla\theta^{\varepsilon}\left(\theta^{\varepsilon}-\theta^{0}\right)\\
	&  & -\varepsilon\left(\tau^{\varepsilon}-\mathbb{T}^{i}\right)\nabla^{\delta}u_{i}^{\varepsilon}\cdot\nabla\theta^{\varepsilon}m^{\varepsilon}\bar{\theta}^{\varepsilon}\cdot\nabla_{x}\theta^{0}\\
	&  & +\mathbb{T}^{i}\left(\nabla^{\delta}u_{i}^{\varepsilon}-\nabla^{\delta}u_{i}^{0}\right)\cdot\nabla\theta^{\varepsilon}\left(\theta^{\varepsilon}-\theta^{0}\right)\\
	&  & -\varepsilon\mathbb{T}^{i}\left(\nabla^{\delta}u_{i}^{\varepsilon}-\nabla^{\delta}u_{i}^{0}\right)\cdot\nabla\theta^{\varepsilon}m^{\varepsilon}\bar{\theta}^{\varepsilon}\cdot\nabla_{x}\theta^{0}\\
	&  & +\mathbb{T}^{i}\left(\nabla\theta^{\varepsilon}-\nabla\theta^{0}\right)\cdot\nabla^{\delta}u_{i}^{0}\left(\theta^{\varepsilon}-\theta^{0}\right)\\
	&  & -\varepsilon\mathbb{T}^{i}\left(\nabla\theta^{\varepsilon}-\nabla\theta^{0}\right)\cdot\nabla^{\delta}u_{i}^{0}m^{\varepsilon}\bar{\theta}^{\varepsilon}\cdot\nabla_{x}\theta^{0}\\
	& = & \sum_{k=1}^{6}\mathcal{I}_{4}^{k}.
\end{eqnarray*}

To be able to estimate $\mathcal{I}_4$, we need to ensure the boundedness of each of the terms $\int_{\Omega^{\varepsilon}}\mathcal{I}_4^{k_i}$ for $k_i\in \left\{1,...,6\right\}$ and $i\in \left\{1,...,N\right\}$. We obtain: 
\begin{eqnarray}
\int_{\Omega^{\varepsilon}}\left|\mathcal{I}_{4}^{2}\right|dx
& \le & \varepsilon\left\Vert \nabla^{\delta}u_{i}^{\varepsilon}\cdot\nabla\theta^{\varepsilon}\right\Vert _{L^{2}\left(\Omega^{\varepsilon}\right)}\left\Vert \left(\tau^{\varepsilon}-\mathbb{T}^{i}\right)m^{\varepsilon}\bar{\theta}\left(\frac{x}{\varepsilon}\right)\cdot\nabla_{x}\theta^{0}\right\Vert _{L^{2}\left(\Omega^{\varepsilon}\right)}\nonumber \\
& \le & \varepsilon\left\Vert u_{i}^{\varepsilon}\right\Vert _{L^{\infty}\left(\Omega^{\varepsilon}\right)}\left\Vert \nabla\theta^{\varepsilon}\right\Vert _{\left[L^{2}\left(\Omega^{\varepsilon}\right)\right]^{d}}\left\Vert \bar{\theta}\right\Vert _{L^{\infty}\left(\Omega^{\varepsilon};C(Y_{1})\right)}\left\Vert \theta^{0}\right\Vert _{W^{1,\infty}(\Omega^{\varepsilon})}\left\Vert \tau^{\varepsilon}-\mathbb{T}^{i}\right\Vert _{L^{2}\left(\Omega^{\varepsilon}\right)},\label{eq:3.12}
\end{eqnarray}
and
\begin{eqnarray}
\int_{\Omega^{\varepsilon}}\left|\mathcal{I}_{4}^{4}\right|dx & \le & \frac{\varepsilon}{2}\left|\mathbb{T}^{i}\right|\left\Vert \left(\nabla^{\delta}u_{i}^{\varepsilon}-\nabla^{\delta}u_{i}^{0}\right)\cdot\nabla\theta^{\varepsilon}\right\Vert _{L^{2}\left(\Omega^{\varepsilon}\right)}\left\Vert m^{\varepsilon}\bar{\theta}\left(\frac{x}{\varepsilon}\right)\cdot\nabla_{x}\theta^{0}\right\Vert _{L^{2}\left(\Omega^{\varepsilon}\right)}\nonumber \\
& \le & \frac{\varepsilon}{2}\left|\mathbb{T}^{i}\right|C_{\delta}^{2}\left\Vert u_{i}^{\varepsilon}-u_{i}^{0}\right\Vert _{L^{2}\left(\Omega^{\varepsilon}\right)}\left\Vert \nabla\theta^{\varepsilon}\right\Vert _{\left[L^{2}\left(\Omega^{\varepsilon}\right)\right]^{d}}\left\Vert \bar{\theta}\right\Vert _{L^{\infty}\left(\Omega^{\varepsilon};C(Y_{1})\right)}\left\Vert \theta^{0}\right\Vert _{W^{1,\infty}(\Omega^{\varepsilon})}.\label{3.13}
\end{eqnarray}

Furthermore, we estimate
\begin{eqnarray}
\int_{\Omega^{\varepsilon}}\left|\mathcal{I}_{4}^{1}\right|dx & \le & \left\Vert \tau^{\varepsilon}-\mathbb{T}^{i}\right\Vert _{L^{2}\left(\Omega^{\varepsilon}\right)}\left\Vert \nabla^{\delta}u_{i}^{\varepsilon}\cdot\nabla\theta^{\varepsilon}\right\Vert _{L^{2}\left(\Omega^{\varepsilon}\right)}\left\Vert \theta^{\varepsilon}-\theta^{0}\right\Vert _{L^{\infty}\left(\Omega^{\varepsilon}\right)}\nonumber \\
& \le & C_{\delta}\left\Vert \tau^{\varepsilon}-\mathbb{T}^{i}\right\Vert _{L^{2}\left(\Omega^{\varepsilon}\right)}\left\Vert u_{i}^{\varepsilon}\right\Vert _{L^{\infty}\left(\Omega^{\varepsilon}\right)}\left\Vert \nabla\theta^{\varepsilon}\right\Vert _{\left[L^{2}\left(\Omega^{\varepsilon}\right)\right]^{d}}\left(\left\Vert \theta^{\varepsilon}\right\Vert _{L^{\infty}\left(\Omega^{\varepsilon}\right)}+\left\Vert \theta^{0}\right\Vert _{W^{1,\infty}\left(\Omega^{\varepsilon}\right)}\right),
\end{eqnarray}
and by Young's inequality, it yields
\begin{eqnarray}
\int_{\Omega^{\varepsilon}}\left|\mathcal{I}_{4}^{3}\right|dx & \le & \frac{\left|\mathbb{T}^{i}\right|^{2}}{2}C_{\delta}^{2}\left\Vert \nabla^{\delta}u_{i}^{\varepsilon}-\nabla^{\delta}u_{i}^{0}\right\Vert _{L^{\infty}\left(\Omega^{\varepsilon}\right)}^{2}\left\Vert \nabla\theta^{\varepsilon}\right\Vert _{\left[L^{2}\left(\Omega^{\varepsilon}\right)\right]^{d}}^{2}+\frac{1}{2}\left\Vert \theta^{\varepsilon}-\theta^{0}\right\Vert _{L^{2}\left(\Omega^{\varepsilon}\right)}^{2}\nonumber \\
& \le & \frac{\left|\mathbb{T}^{i}\right|^{2}}{2}C_{\delta}^{4}\left\Vert \nabla\theta^{\varepsilon}\right\Vert _{\left[L^{2}\left(\Omega^{\varepsilon}\right)\right]^{d}}^{2}\left\Vert u_{i}^{\varepsilon}-u_{i}^{0}\right\Vert _{L^{2}\left(\Omega^{\varepsilon}\right)}^{2}+\frac{1}{2}\left\Vert \theta^{\varepsilon}-\theta^{0}\right\Vert _{L^{2}\left(\Omega^{\varepsilon}\right)}^{2},
\end{eqnarray}
and
\begin{eqnarray}
\int_{\Omega^{\varepsilon}}\left|\mathcal{I}_{4}^{5}\right|dx & \le & \frac{\left|\mathbb{T}_{i}\right|^{2}}{2}\left\Vert \left(\nabla\theta^{\varepsilon}-\nabla\theta^{0}\right)\cdot\nabla^{\delta}u_{i}^{0}\right\Vert _{L^{2}\left(\Omega^{\varepsilon}\right)}^{2}+\frac{1}{2}\left\Vert \theta^{\varepsilon}-\theta^{0}\right\Vert _{L^{2}\left(\Omega^{\varepsilon}\right)}^{2}\nonumber \\
& \le & \frac{\left|\mathbb{T}_{i}\right|^{2}}{2}C_{\delta}^{2}\left\Vert u_{i}^{0}\right\Vert _{L^{\infty}\left(\Omega^{\varepsilon}\right)}^{2}\left\Vert \nabla\theta^{\varepsilon}-\nabla\theta^{0}\right\Vert _{\left[L^{2}\left(\Omega^{\varepsilon}\right)\right]^{d}}^{2}+\frac{1}{2}\left\Vert \theta^{\varepsilon}-\theta^{0}\right\Vert _{L^{2}\left(\Omega^{\varepsilon}\right)}^{2},
\end{eqnarray}
\begin{equation}
\int_{\Omega^{\varepsilon}}\left|\mathcal{I}_{4}^{6}\right|dx\le\frac{\varepsilon\left|\mathbb{T}_{i}\right|^{2}}{2}\left\Vert u_{i}^{0}\right\Vert _{L^{\infty}\left(\Omega^{\varepsilon}\right)}^{2}\left\Vert \nabla\theta^{\varepsilon}-\nabla\theta^{0}\right\Vert _{\left[L^{2}\left(\Omega^{\varepsilon}\right)\right]^{d}}^{2}+\frac{\varepsilon}{2}\left\Vert \bar{\theta}\right\Vert^2 _{L^{\infty}\left(\Omega^{\varepsilon};C(\bar{Y_{1}})\right)}\left\Vert \theta^{0}\right\Vert^2 _{W^{1,\infty}(\Omega^\varepsilon)}.\label{3.17}
\end{equation}

Remark that the first integral in $\mathcal{J}_{4}^{i}$
can be estimated similarly. On top of that, observe that we can find constants
$C_{R_{i}}>0$ (independent of $\varepsilon$) such that
\[
\left\Vert R_{i}\left(u^{\varepsilon}\right)-R_{i}\left(u^{0}\right)\right\Vert _{L^{2}\left(\Omega^{\varepsilon}\right)}\le C_{R_{i}}\sum_{j=1}^{N}\left\Vert u_{j}^{\varepsilon}-u_{j}^{0}\right\Vert _{L^{2}\left(\Omega^{\varepsilon}\right)}\quad\mbox{for}\;i\in\left\{ 1,\ldots,N\right\},
\]
in which the constants $C_{R_i}$ depend on the $L^{\infty}$-bounds of the concentrations $u^{\varepsilon},u^0$ as discussed in \cite[Section 5]{KM16}.

The estimate on the second integral of $\mathcal{J}_{4}^{i}$
can be computed directly. Note that for $i\in\left\{ 1,\ldots,N\right\} $, we have:
\begin{eqnarray*}
	\left(R_{i}\left(u^{\varepsilon}\right)-R_{i}\left(u^{0}\right)\right)\phi_{i} & = & \left(R_{i}\left(u^{\varepsilon}\right)-R_{i}\left(u^{0}\right)\right)\left(u_{i}^{\varepsilon}-u_{i}^{0}\right)\\
	&  & -\varepsilon\left(R_{i}\left(u^{\varepsilon}\right)-R_{i}\left(u^{0}\right)\right)m^{\varepsilon}\bar{u}_{i}\left(\frac{x}{\varepsilon}\right)\cdot\nabla_{x}u_{i}^{0}.
\end{eqnarray*}

This gives
\begin{eqnarray}
\int_{\Omega^{\varepsilon}}\left(R_{i}\left(u^{\varepsilon}\right)-R_{i}\left(u^{0}\right)\right)\phi_{i}dx & \le & C_{R_{i}}\sum_{j=1}^{N}\left\Vert u_{j}^{\varepsilon}-u_{j}^{0}\right\Vert _{L^{2}\left(\Omega^{\varepsilon}\right)}\left(\left\Vert u_{i}^{\varepsilon}-u_{i}^{0}\right\Vert _{L^{2}\left(\Omega^{\varepsilon}\right)}+\right.\nonumber \\
&  & \left.+\varepsilon\left\Vert \bar{u}_{i}\right\Vert _{L^{\infty}\left(\Omega^{\varepsilon};C(\bar{Y_{1}})\right)}\left\Vert u_{i}^{0}\right\Vert _{W^{1,\infty}(\Omega^{\varepsilon})}\right).\label{3.18}
\end{eqnarray}

Collecting the estimates (\ref{eq:3.8}), (\ref{3.9}),
(\ref{3.10}), (\ref{3.11}), (\ref{eq:3.12})-(\ref{3.17}) and (\ref{3.18}),
we obtain:\\
$\displaystyle{
	\left\Vert \nabla\left(\theta^{\varepsilon}-\theta_{1}^{\varepsilon}\right)\left(t\right)\right\Vert _{\left[L^{2}\left(\Omega^{\varepsilon}\right)\right]^{d}}^{2}+\sum_{i=1}^{N}\left\Vert \nabla\left(u_{i}^{\varepsilon}-u_{i,1}^{\varepsilon}\right)\left(t\right)\right\Vert _{\left[L^{2}\left(\Omega^{\varepsilon}\right)\right]^{d}}^{2}
}
$\\
$
\le C\left(\varepsilon^{2}+\varepsilon\right){\displaystyle +C\varepsilon\left(\left\Vert \theta^{\varepsilon}\left(t\right)-\theta^{0}\left(t\right)\right\Vert _{L^{2}\left(\Omega^{\varepsilon}\right)}+\left\Vert \nabla\left(\theta^{\varepsilon}-\theta_{1}^{\varepsilon}\right)\left(t\right)\right\Vert _{\left[L^{2}\left(\Omega^{\varepsilon}\right)\right]^{d}}+C(1+\varepsilon)\right)}
$\\
$\displaystyle{\quad
	+C\varepsilon\sum_{i=1}^{N}\left(\left\Vert u_{i}^{\varepsilon}\left(t\right)-u_{i}^{0}\left(t\right)\right\Vert _{L^{2}\left(\Omega^{\varepsilon}\right)}+\left\Vert \nabla\left(u_{i}^{\varepsilon}-u_{i,1}^{\varepsilon}\right)\left(t\right)\right\Vert _{\left[L^{2}\left(\Omega^{\varepsilon}\right)\right]^{d}}+C(1+\varepsilon)\right)
}
$\\
$\displaystyle{\quad
	+C\left(\left\Vert \tau^{\varepsilon}-\mathbb{T}^{i}\right\Vert _{L^{2}\left(\Omega^{\varepsilon}\right)}\left\Vert \theta^{\varepsilon}\left(t\right)-\theta^{0}\left(t\right)\right\Vert _{L^{2}\left(\Omega^{\varepsilon}\right)}+\sum_{i=1}^{N}\left\Vert \rho_{i}^{\varepsilon}-\mathbb{F}^{i}\right\Vert _{L^{2}\left(\Omega^{\varepsilon}\right)}\left\Vert u_{i}^{\varepsilon}\left(t\right)-u_{i}^{0}\left(t\right)\right\Vert _{L^{2}\left(\Omega^{\varepsilon}\right)}\right)
}
$\\
$\displaystyle{\quad
	+C\varepsilon\left(\sum_{i=1}^{N}\left\Vert u_{i}^{\varepsilon}\left(t\right)-u_{i}^{0}\left(t\right)\right\Vert _{L^{2}\left(\Omega^{\varepsilon}\right)}+\left\Vert \theta^{\varepsilon}\left(t\right)-\theta^{0}\left(t\right)\right\Vert _{L^{2}\left(\Omega^{\varepsilon}\right)}\right)
}
$\\
$\displaystyle{\quad
	+C\left(\sum_{i=1}^{N}\left\Vert u_{i}^{\varepsilon}\left(t\right)-u_{i}^{0}\left(t\right)\right\Vert _{L^{2}\left(\Omega^{\varepsilon}\right)}^{2}+\left\Vert \theta^{\varepsilon}\left(t\right)-\theta^{0}\left(t\right)\right\Vert _{L^{2}\left(\Omega^{\varepsilon}\right)}^{2}\right)
}
$\\
$\displaystyle{\quad
	+C\varepsilon\left(\left\Vert \nabla\left(\theta^{\varepsilon}-\theta^{0}\right)\left(t\right)\right\Vert _{\left[L^{2}\left(\Omega^{\varepsilon}\right)\right]^{d}}^{2}+\sum_{i=1}^{N}\left\Vert \nabla\left(u_{i}^{\varepsilon}-u_{i}^{0}\right)\left(t\right)\right\Vert _{\left[L^{2}\left(\Omega^{\varepsilon}\right)\right]^{d}}^{2}\right)+C\varepsilon.}
$

Notably, Theorem \ref{thm:pep} provides us that the $L^2$-error estimates between the Soret and Dufour coefficients and their homogenized (averaged) versions, i.e. $\left\Vert \tau^{\varepsilon}-\mathbb{T}^{i}\right\Vert _{L^{2}\left(\Omega^{\varepsilon}\right)}$ and $\left\Vert \rho_{i}^{\varepsilon}-\mathbb{F}^{i}\right\Vert _{L^{2}\left(\Omega^{\varepsilon}\right)}$ are of the order $\mathcal{O}(\varepsilon^{1/2})$. It thus yields that
\begin{align}
\left\Vert \nabla\left(\theta^{\varepsilon}-\theta_{1}^{\varepsilon}\right)\left(t\right)\right\Vert _{\left[L^{2}\left(\Omega^{\varepsilon}\right)\right]^{d}}^{2}+\sum_{i=1}^{N}\left\Vert \nabla\left(u_{i}^{\varepsilon}-u_{i,1}^{\varepsilon}\right)\left(t\right)\right\Vert _{\left[L^{2}\left(\Omega^{\varepsilon}\right)\right]^{d}}^{2}\qquad\qquad\qquad\qquad \nonumber\\
\le C\left(\varepsilon^{2}+\varepsilon\right)+C\varepsilon\left(\left\Vert \theta^{\varepsilon}\left(t\right)-\theta^{0}\left(t\right)\right\Vert _{L^{2}\left(\Omega^{\varepsilon}\right)}+\left\Vert \nabla\left(\theta^{\varepsilon}-\theta_{1}^{\varepsilon}\right)\left(t\right)\right\Vert _{\left[L^{2}\left(\Omega^{\varepsilon}\right)\right]^{d}}\right) \nonumber\\
+C\varepsilon\sum_{i=1}^{N}\left(\left\Vert u_{i}^{\varepsilon}\left(t\right)-u_{i}^{0}\left(t\right)\right\Vert _{L^{2}\left(\Omega^{\varepsilon}\right)}+\left\Vert \nabla\left(u_{i}^{\varepsilon}-u_{i,1}^{\varepsilon}\right)\left(t\right)\right\Vert _{\left[L^{2}\left(\Omega^{\varepsilon}\right)\right]^{d}}\right)\qquad \nonumber\\
+C\varepsilon^{1/2}\left(\left\Vert \theta^{\varepsilon}\left(t\right)-\theta^{0}\left(t\right)\right\Vert _{L^{2}\left(\Omega^{\varepsilon}\right)}+\sum_{i=1}^{N}\left\Vert u_{i}^{\varepsilon}\left(t\right)-u_{i}^{0}\left(t\right)\right\Vert _{L^{2}\left(\Omega^{\varepsilon}\right)}\right)\qquad \; \nonumber\\
+C\left(\sum_{i=1}^{N}\left\Vert u_{i}^{\varepsilon}\left(t\right)-u_{i}^{0}\left(t\right)\right\Vert _{L^{2}\left(\Omega^{\varepsilon}\right)}^{2}+\left\Vert \theta^{\varepsilon}\left(t\right)-\theta^{0}\left(t\right)\right\Vert _{L^{2}\left(\Omega^{\varepsilon}\right)}^{2}\right)\qquad \qquad \nonumber\\
+C\varepsilon\left(\left\Vert \nabla\left(\theta^{\varepsilon}-\theta_{1}^{\varepsilon}\right)\left(t\right)\right\Vert _{\left[L^{2}\left(\Omega^{\varepsilon}\right)\right]^{d}}^{2}+\sum_{i=1}^{N}\left\Vert \nabla\left(u_{i}^{\varepsilon}-u_{i,1}^{\varepsilon}\right)\left(t\right)\right\Vert _{\left[L^{2}\left(\Omega^{\varepsilon}\right)\right]^{d}}^{2}\right).
\label{3.19-1}
\end{align}

It now remains to estimate the second term on the right-hand side of \eqref{eq:3.3}-\eqref{eq:3.4}. In fact, integrating by parts gives
\begin{align*}
\int_{0}^{t}\int_{\Omega^{\varepsilon}}m^{\varepsilon}\partial_{t}\left(u_{i}^{\varepsilon}-u_{i}^{0}\right)\bar{u}_{i}\left(\frac{x}{\varepsilon}\right)\cdot\nabla_{x}u_{i}^{0}\left(s,x\right)dxds= & \left.\int_{\Omega^{\varepsilon}}m^{\varepsilon}\left(u_{i}^{\varepsilon}-u_{i}^{0}\right)\bar{u}_{i}\left(\frac{x}{\varepsilon}\right)\cdot\nabla_{x}u_{i}^{0}\left(s,x\right)dx\right|_{s=0}^{s=t}\\ -
& \int_{0}^{t}\int_{\Omega^{\varepsilon}}m^{\varepsilon}\left(u_{i}^{\varepsilon}-u_{i}^{0}\right)\bar{u}_{i}\left(\frac{x}{\varepsilon}\right)\cdot\nabla_{x}\partial_{t}u_{i}^{0}\left(s,x\right)dxds.
\end{align*}

We then observe that
\begin{align*}
\varepsilon\left|\int_{\Omega^{\varepsilon}}m^{\varepsilon}\left[\left(u_{i}^{\varepsilon}-u_{i}^{0}\right)-\left(u_{i}^{\varepsilon}\left(0\right)-u_{i}^{0}\left(0\right)\right)\right]\bar{u}_{i}^{\varepsilon}\cdot\nabla_{x}u_{i}^{0}\left(t,x\right)dx\right|\\
\le C\varepsilon\left(\left\Vert u_{i}^{\varepsilon}\left(t\right)-u_{i}^{0}\left(t\right)\right\Vert _{L^{2}\left(\Omega^{\varepsilon}\right)}+\left\Vert u_{i}^{\varepsilon,0}-u_{i}^{0,0}\right\Vert _{L^{2}\left(\Omega^{\varepsilon}\right)}\right),
\end{align*}
and hence,
\begin{align*}
\varepsilon\left|\int_{\Omega^{\varepsilon}}m^{\varepsilon}\left[\left(\theta^{\varepsilon}-\theta^{0}\right)-\left(\theta^{\varepsilon}\left(0\right)-\theta^{0}\left(0\right)\right)\right]\bar{\theta}^{\varepsilon}\cdot\nabla_{x}\theta^{0}\left(t,x\right)dx\right|\\
\le C\varepsilon\left(\left\Vert \theta^{\varepsilon}\left(t\right)-\theta^{0}\left(t\right)\right\Vert _{L^{2}\left(\Omega^{\varepsilon}\right)}+\left\Vert \theta^{\varepsilon,0}-\theta^{0,0}\right\Vert _{L^{2}\left(\Omega^{\varepsilon}\right)}\right).
\end{align*}

For all $t\in (0,T]$, we set
\begin{align*}
w_{1}\left(t\right) & =\left\Vert \theta^{\varepsilon}\left(t\right)-\theta^{0}\left(t\right)\right\Vert _{L^{2}\left(\Omega^{\varepsilon}\right)}^{2}+\sum_{i=1}^{N}\left\Vert u_{i}^{\varepsilon}\left(t\right)-u_{i}^{0}\left(t\right)\right\Vert _{L^{2}\left(\Omega^{\varepsilon}\right)}^{2},\\
w_{2}\left(t\right) & =\left\Vert \nabla\left(\theta^{\varepsilon}-\theta_{1}^{\varepsilon}\right)\left(t\right)\right\Vert _{\left[L^{2}\left(\Omega^{\varepsilon}\right)\right]^{d}}^{2}+\sum_{i=1}^{N}\left\Vert \nabla\left(u_{i}^{\varepsilon}-u_{i,1}^{\varepsilon}\right)\left(t\right)\right\Vert _{\left[L^{2}\left(\Omega^{\varepsilon}\right)\right]^{d}}^{2},\\
w_{0} & =\left\Vert \theta^{\varepsilon,0}-\theta^{0,0}\right\Vert _{L^{2}\left(\Omega^{\varepsilon}\right)}^{2}+\sum_{i=1}^{N}\left\Vert u_{i}^{\varepsilon,0}-u_{i}^{0,0}\right\Vert _{L^{2}\left(\Omega^{\varepsilon}\right)}^{2}.
\end{align*}

Then, when integrating \eqref{3.19-1} and \eqref{eq:3.3}-\eqref{eq:3.4} from 0 to $t$, we are led to the following Gronwall-like estimate
\[
w_{1}\left(t\right)+\int_{0}^{t}w_{2}\left(s\right)ds\le C\left(\varepsilon^{2}+\varepsilon+\left(1+\varepsilon\right)w_{0}+\varepsilon\int_{0}^{t}w_{1}\left(s\right)ds\right),
\]
which can be rewritten as
\begin{equation}
w_{1}\left(t\right)+\int_{0}^{t}w_{2}\left(s\right)ds\le C\left(\varepsilon+\left(1+\varepsilon\right)w_{0}\right)e^{C\varepsilon t}
\quad \mbox{for}\; t\in [0,T].
\label{eq:corrector1}
\end{equation}

Finally, we turn our attention to the corrector estimate for $v_{i}^{\varepsilon}$. For $i\in\left\{ 1,...,N\right\} $ we consider the equation
for the reconstruction $v_{i,0}^{\varepsilon}=v_{i}^{0}$, obtained
from \eqref{eq:2.14}, with the test function $\psi_{i}\in L^{2}\left(\Gamma^{\varepsilon}\right)$
and integrate the resulting equation over $\Gamma^{\varepsilon}$
to get
\begin{equation}
\varepsilon\int_{\Gamma^{\varepsilon}}\partial_{t}v_{i}^{0}\psi_{i}dS_{\varepsilon}=\varepsilon\int_{\Gamma^{\varepsilon}}\left(A_{i}u_{i}^{0}-B_{i}v_{i}^{0}\right)\psi_{i}dS_{\varepsilon}.
\label{eq:lasteq}
\end{equation}

Then, we find the difference equation for the micro concentration
$v_{i}^{\varepsilon}$ and the reconstruction $v_{i}^{0}$ by subtracting
the third equation of \eqref{eq:weak1} and \eqref{eq:lasteq}, provided that
\begin{align*}
\varepsilon\int_{\Gamma^{\varepsilon}}\partial_{t}\left(v_{i}^{\varepsilon}-v_{i}^{0}\right)\psi_{i}dS_{\varepsilon} & =\varepsilon\int_{\Gamma^{\varepsilon}}\left(a_{i}^{\varepsilon}u_{i}^{\varepsilon}-A_{i}u_{i}^{0}\right)\psi_{i}dS_{\varepsilon}-\varepsilon\int_{\Gamma^{\varepsilon}}\left(b_{i}^{\varepsilon}v_{i}^{\varepsilon}-B_{i}v_{i}^{0}\right)\psi_{i}dS_{\varepsilon}\\
& =\varepsilon\int_{\Gamma^{\varepsilon}}\left[a_{i}^{\varepsilon}\left(u_{i}^{\varepsilon}-u_{i}^{0}\right)+\left(a_{i}^{\varepsilon}-A_{i}\right)u_{i}^{0}\right]\psi_{i}dS_{\varepsilon}\\
& -\varepsilon\int_{\Gamma^{\varepsilon}}\left[b_{i}^{\varepsilon}\left(v_{i}^{\varepsilon}-v_{i}^{0}\right)+\left(b_{i}^{\varepsilon}-B_{i}\right)v_{i}^{0}\right]\psi_{i}dS_{\varepsilon}.
\end{align*}

Hereby, we choose $\psi_{i}=v_{i}^{\varepsilon}-v_{i}^{0}$ to obtain
the following estimate
\begin{align}
\frac{\varepsilon}{2}\frac{d}{dt}\left\Vert v_{i}^{\varepsilon}-v_{i}^{0}\right\Vert _{L^{2}\left(\Gamma^{\varepsilon}\right)}^{2} & \le C\varepsilon\left(\left\Vert u_{i}^{\varepsilon}-u_{i}^{0}\right\Vert _{L^{2}\left(\Gamma^{\varepsilon}\right)}^{2}+\left\Vert v_{i}^{\varepsilon}-v_{i}^{0}\right\Vert _{L^{2}\left(\Gamma^{\varepsilon}\right)}^{2}\right)\nonumber \\ 
& +\varepsilon\int_{\Gamma^{\varepsilon}}\left|a_{i}^{\varepsilon}-A_{i}\right|\left|u_{i}^{0}\right|\left|v_{i}^{\varepsilon}-v_{i}^{0}\right|dS_{\varepsilon}+\varepsilon\int_{\Gamma^{\varepsilon}}\left|b_{i}^{\varepsilon}-B_{i}\right|\left|v_{i}^{0}\right|\left|v_{i}^{\varepsilon}-v_{i}^{0}\right|dS_{\varepsilon}.
\label{eq:nextin}
\end{align}

Since $\Omega^{\varepsilon}$ is a Lipschitz domain, we recall the
trace embedding $H^{1}\left(\Omega^{\varepsilon}\right)\subset L^{q}\left(\partial\Omega^{\varepsilon}\right)$
which holds for $1\le q\le2_{\partial\Omega^{\varepsilon}}^{*}$ where
$2_{\partial\Omega^{\varepsilon}}^{*}=2\left(d-1\right)/\left(d-2\right)$
if $d\ge3$, and $2_{\partial\Omega^{\varepsilon}}^{*}=\infty$ if
$d=2$ (cf. \cite{Fan08}). Therefore, when the two-dimensional case is concentrated, we continue
to estimate \eqref{eq:nextin}, as follows:
\begin{align*}
\frac{\varepsilon}{2}\frac{d}{dt}\sum_{i=1}^{N}\left\Vert v_{i}^{\varepsilon}-v_{i}^{0}\right\Vert _{L^{2}\left(\Gamma^{\varepsilon}\right)}^{2} & \le C\varepsilon\left(\sum_{i=1}^{N}\left\Vert u_{i}^{\varepsilon}-u_{i}^{0}\right\Vert _{L^{2}\left(\Gamma^{\varepsilon}\right)}^{2}+\sum_{i=1}^{N}\left\Vert v_{i}^{\varepsilon}-v_{i}^{0}\right\Vert _{L^{2}\left(\Gamma^{\varepsilon}\right)}^{2}\right)\\
& +C\varepsilon\left(\sum_{i=1}^{N}\left\Vert a_{i}^{\varepsilon}-A_{i}\right\Vert _{L^{2}\left(\Gamma^{\varepsilon}\right)}^{2}+\sum_{i=1}^{N}\left\Vert b_{i}^{\varepsilon}-B_{i}\right\Vert _{L^{2}\left(\Gamma^{\varepsilon}\right)}^{2}\right).
\end{align*}

Observe that using the trace inequality \eqref{eq:tracein} for the difference norms $\left\Vert a_{i}^{\varepsilon}-A_{i}\right\Vert _{L^{2}\left(\Gamma^{\varepsilon}\right)}$,
$\left\Vert b_{i}^{\varepsilon}-B_{i}\right\Vert _{L^{2}\left(\Gamma^{\varepsilon}\right)}$
and $\left\Vert u_{i}^{\varepsilon}-u_{i}^{0}\right\Vert _{L^{2}\left(\Gamma^{\varepsilon}\right)}$ together with Lemma \ref{thm:pep} and \eqref{eq:corrector1} 
gives
\begin{equation}
\frac{\varepsilon}{2}\frac{d}{dt}\sum_{i=1}^{N}\left\Vert v_{i}^{\varepsilon}-v_{i}^{0}\right\Vert _{L^{2}\left(\Gamma^{\varepsilon}\right)}^{2}\le C\max\left\{ \varepsilon,\varepsilon^{\gamma}\right\} +C\varepsilon\sum_{i=1}^{N}\left\Vert v_{i}^{\varepsilon}-v_{i}^{0}\right\Vert _{L^{2}\left(\Gamma^{\varepsilon}\right)}^{2}.
\label{eq:sapxongroi}
\end{equation}
Note herein that the gradient norms are ignored when applying the trace inequality to the differences. It is simply because that they are of the order $\mathcal{O}(\varepsilon^2)$ by their own regularity.

Henceforward, we apply the Gronwall inequality to \eqref{eq:sapxongroi} and obtain
\[
\varepsilon\sum_{i=1}^{N}\left\Vert v_{i}^{\varepsilon}-v_{i}^{0}\right\Vert _{L^{2}\left(\Gamma^{\varepsilon}\right)}^{2}\le C\max\left\{ \varepsilon,\varepsilon^{\gamma}\right\} e^{C\varepsilon t}.
\]

In the same manner, if $d\ge 3$ is applied, we can bound the absolute differences $\left|a_{i}^{\varepsilon} - A_i\right|$ and $\left|b_{i}^{\varepsilon} - B_i\right|$ in \eqref{eq:nextin} from above by a constant $C$ independent of $\varepsilon$ (by $\left(\text{A}_{1}\right)$) and then get back the estimate \eqref{eq:sapxongroi}.

This completes the proof of Theorem \ref{mainthm:1}.


\section{Conclusions}\label{sec:conclu}
In this work, we have presented corrector estimates for  the homogenization limit for a thermo-diffusion system with Smoluchowski interactions coupled with a system of differential equations, posed in a perforated domain.  This type of error-control justifies the formal homogenization asymptotics obtained in \cite{KMK15} and completes the convergence result in \cite{KAM14} by giving convergence rates. This is done using the concept of macroscopic reconstruction together with fine integral estimates on the solution and oscillating coefficients. Our working technique can be applied to a larger class of coupled nonlinear systems of partial differential equations posed in perforated media.


\section*{Acknowledgment}
This work was initiated when V.A.K. visited the Department of Mathematics and Computer Science of Karlstad University, Sweden. This work is dedicated to the memory of his beloved father. A.M. thanks NWO MPE `Theoretical estimates of heat losses in geothermal wells" (grant nr. 657.014.004) for funding.

\bibliographystyle{plain}
\bibliography{mybib}

\end{document}